\documentclass{amsart}
\usepackage[utf8]{inputenc}
\usepackage[margin=1.25in]{geometry}
\usepackage{amsmath}
\usepackage{amssymb}
\usepackage{mathtools}
\usepackage{amsthm}
\usepackage{hyperref}
\usepackage{cleveref}
\usepackage{amsfonts}
\usepackage{shuffle}
\usepackage{graphicx}
\usepackage{tikz-cd}
\usepackage{enumerate}
\usepackage{verbatim}
\usepackage{mathrsfs}
\usepackage{ytableau}
\usepackage{ifthen}

\newtheorem{theorem}{Theorem}[section]
\newtheorem{lemma}[theorem]{Lemma}
\newtheorem{proposition}[theorem]{Proposition}
\newtheorem{corollary}[theorem]{Corollary}

\theoremstyle{definition}
\newtheorem{definition}[theorem]{Definition}

\newtheorem{claim}[theorem]{Claim}

\theoremstyle{remark}
\newtheorem{remark}[theorem]{Remark}
\newtheorem{example}[theorem]{Example}

\numberwithin{equation}{section}

\DeclareMathOperator{\trip}{\mathsf{trip}}

\DeclareMathOperator{\aexc}{\mathrm{Aexc}}

\DeclareMathOperator{\prom}{\mathsf{prom}}
\DeclareMathOperator{\promotion}{\mathcal{P}}

\DeclareMathOperator{\evacuation}{\mathcal{E}}

\DeclareMathOperator{\SL}{SL}
\newcommand{\fsl}{\mathfrak{sl}}
\DeclareMathOperator{\Inv}{\mathsf{Inv}}
\DeclareMathOperator{\Hom}{\mathrm{Hom}}

\DeclareMathOperator{\SYT}{\mathsf{SYT}}

\DeclareMathOperator{\dis}{\mathfrak{d}}

\newcommand{\newword}[1]{\emph{#1}}

\title{Web bases in degree two from hourglass plabic graphs}

\author[Gaetz]{Christian Gaetz}
\address[Gaetz]{Department of Mathematics, University of California, Berkeley, CA, USA.}
\email{gaetz@berkeley.edu}

\author[Pechenik]{Oliver Pechenik}
\address[Pechenik]{Department of Combinatorics \& Optimization, University of Waterloo, Waterloo, ON, Canada.}
\email{oliver.pechenik@uwaterloo.ca}

\author[Pfannerer]{Stephan Pfannerer}  
\address[Pfannerer]{Department of Combinatorics \& Optimization, University of Waterloo, Waterloo, ON, Canada.}
\email{math@pfannerer-mittas.net}

\author[Striker]{Jessica Striker}
\address[Striker]{Department of Mathematics, North Dakota State University, Fargo, ND, USA.}
\email{jessica.striker@ndsu.edu}

\author[Swanson]{Joshua P. Swanson}
\address[Swanson]{Department of Mathematics, University of Southern California, Los Angeles, CA, USA.}
\email{swansonj@usc.edu}

\date{\today}

\AtBeginDocument{%
   \def\MR#1{}
}

\begin{document}

\begin{abstract}
    Webs give a diagrammatic calculus for spaces of $U_q(\fsl_r)$-tensor invariants, but intrinsic characterizations of web bases are only known in certain cases. Recently, we introduced \emph{hourglass plabic graphs} to give the first such $U_q(\fsl_4)$-web bases. Separately, Fraser introduced a web basis for \emph{Pl\"{u}cker degree two} representations of arbitrary $U_q(\fsl_r)$. Here, we show that Fraser's basis agrees with that predicted by the hourglass plabic graph framework and give an intrinsic characterization of the resulting webs. A further compelling feature with many applications is that our bases exhibit \emph{rotation-invariance}. Together with the results of our earlier paper, this implies that hourglass plabic graphs give a uniform description of all known rotation-invariant $U_q(\fsl_r)$-web bases. Moreover, this provides a single combinatorial model simultaneously generalizing the Tamari lattice, the alternating sign matrix lattice, and the lattice of plane partitions. As a part of our argument, we develop properties of \emph{square faces} in arbitrary hourglass plabic graphs, a key step in our program towards general $U_q(\fsl_r)$-web bases.
\end{abstract}

\maketitle

\section{Introduction}

\emph{Web bases} were introduced by Kuperberg \cite{Kuperberg} to give a diagrammatic calculus for the representation theory of the quantum groups $U_q(\fsl_2)$ and $U_q(\fsl_3)$, with applications to quantum link invariants. Kuperberg's bases have since found further application in areas as diverse as the geometric Satake correspondence \cite[Thm.~1.4]{Fontaine-Kamnitzer-Kuperberg}, cluster algebras \cite{Fomin-Pylyavskyy-advances}, cyclic sieving \cite{Petersen-Pylyavskyy-Rhoades}, quantum topology \cite{Costantino.Le, Higgins}, and dimer models \cite{Douglas-Kenyon-Shi}. A major challenge has been to find generalizations of these bases for quantum groups of higher rank. 

In \cite{Gaetz.Pechenik.Pfannerer.Striker.Swanson:4row}, we introduced \emph{hourglass plabic graphs} (see \S\ref{sec:hourglass-background} below) to construct a web basis for spaces of $U_q(\fsl_r)$-invariants for $r\leq 4$, characterized intrinsically by the \emph{fully reduced} condition. Our construction recovers Kuperberg's bases for $r=2$ and $3$. Here, we extend the theory of hourglass plabic graphs to $U_q(\fsl_r)$ for arbitrary $r$, but restrict to representations of \emph{Pl\"ucker degree two}. 

Fraser \cite{Fraser-2-column} introduced a map $\mathcal{F}$ from $2$-column rectangular standard Young tableaux to webs. The image of $\mathcal{F}$ forms a web basis for
$\Inv_{\SL_r}\left(V^{\otimes 2r}\right) \coloneqq \Hom_{\SL_r}\left(V^{\otimes 2r}, \mathbb{C} \right)$
which coincides with Lusztig's dual canonical basis. We call this setting \emph{Pl\"ucker degree two} since the invariant space can be identified with a graded piece of the homogeneous coordinate ring of the Grassmannian spanned by products of two Pl\"{u}cker coordinates of disjoint support.   

We show that, in this context, the fully reduced hourglass plabic graphs are exactly the outputs of Fraser's construction. This condition thus intrinsically determines a web basis for $\Inv_{\SL_r}\left(V^{\otimes 2r}\right)$, and indeed by \cite[Thm.~2.9]{Gaetz.Pechenik.Pfannerer.Striker.Swanson:4row} for its quantum deformation $\Inv_{U_q(\fsl_r)}(V_q^{\otimes 2r})$. Here $V$ denotes the standard representation $\mathbb{C}^r$ of $\SL_r(\mathbb{C})$ and $V_q$ its quantum deformation, a representation of $U_q(\mathfrak{sl}_r)$.

A web basis is \emph{rotation-invariant} if diagrammatic rotation sends basis elements to basis elements, possibly up to a sign. This rotation of webs corresponds to cyclic permutation of the tensor factors in $V_q^{\otimes 2r}$. The rotation-invariance of the web bases described above is key to many applications.

\begin{theorem}\label{thm:basis}
The collection $\mathscr{B}_q^{2r}$ of tensor invariants of fully reduced $r$-hourglass plabic graphs of Pl\"{u}cker degree two is a rotation-invariant web basis for $\Inv_{U_q(\fsl_r)}(V_q^{\otimes 2r})$.
\end{theorem}

One advantage of our formulation of \Cref{thm:basis} is that it provides an intrinsic graph-theoretic characterization of the basis webs. This stands in contrast to \cite{Fraser-2-column}, where the basis webs are characterized only as the image of $\mathcal{F}$. 

As in \cite{Gaetz.Pechenik.Pfannerer.Striker.Swanson:4row}, we associate to an hourglass plabic graph $G$ a tuple  $\trip_\bullet(G)$ of \emph{trip permutations}, extending the usual \emph{trip permutation} of a plabic graph~\cite{Postnikov-arxiv}. As in \cite{Gaetz.Pechenik.Pfannerer.Striker.Swanson:fluctuating}, we associate to a tableau $T$ a tuple $\prom_\bullet(T)$ of \emph{promotion permutations} encoding the behavior of \emph{jeu de taquin promotion}, building on ideas of Hopkins--Rubey \cite{Hopkins-Rubey}. See \Cref{sec:background} for details.

Our main combinatorial result is as follows; see \Cref{fig:prom_orbit,fig:FraserMap} for an example.  

\begin{theorem} \label{thm:bijection}
The map $\mathcal{F}$ is a  bijection between the set of $r \times 2$ standard Young tableaux and the set of move-equivalence classes of fully reduced $r$-hourglass plabic graphs of standard type and Pl\"ucker degree two. Furthermore, this bijection satisfies $\trip_{\bullet}(\mathcal{F}(T))=\prom_{\bullet}(T)$ and intertwines promotion and evacuation of tableaux with rotation and reflection of hourglass plabic graphs.
\end{theorem}

The move-equivalence classes of \Cref{thm:bijection} are generated by \textit{square moves}. These moves are known to preserve the tensor invariant, so the basis of \Cref{thm:basis} is canonically determined. Furthermore, the move-equivalence classes have interesting combinatorial structure. In particular, we show that the hourglass plabic graph framework unifies the following three classical constructions: the \emph{Tamari lattice}, \emph{alternating sign matrices}, and \emph{plane partitions} in a box; see \Cref{prop:superstandard} and \cite[\S8]{Gaetz.Pechenik.Pfannerer.Striker.Swanson:4row} for details.

The $U_q(\fsl_r)$-web bases for $r \leq 4$ from are also in bijection with tableaux via the same $\trip_{\bullet}=\prom_{\bullet}$ property \cite{Gaetz.Pechenik.Pfannerer.Striker.Swanson:4row}; that work considered rectangular tableaux with at most 4 rows and arbitrarily many columns, whereas this work applies in arbitrarily many rows with at most 2 columns. Thus, \Cref{thm:bijection} also extends this aspect of the hourglass plabic graph framework.

This paper is part of a program aiming at developing web bases for general spaces of $U_q(\fsl_r)$-invariants. In the process of establishing \Cref{thm:basis,thm:bijection}, we develop tools that apply in that generality; see \Cref{thm:square-fully-reduced,thm:square-move-preserves-trip-bullet}.  

The remainder of the paper is organized as follows. \Cref{sec:background} collects background and preliminaries on tableaux and promotion, on hourglass plabic graphs, and on Fraser's map $\mathcal{F}$. In \Cref{sec:trip-prom} we show that $\trip_{\bullet}(\mathcal{F}(T))=\prom_{\bullet}(T)$, establishing part of \Cref{thm:bijection}. \Cref{sec:fraser_is_fr} proves that the graphs $\mathcal{F}(T)$ are fully reduced. \Cref{sec:square_faces} characterizes when square faces are fully reduced and shows that square moves preserve trip permutations. \Cref{sec:fr_is_fraser} completes the proofs of \Cref{thm:basis,thm:bijection} by showing that the fully reduced hourglass plabic graphs are exactly those produced by Fraser's map. These arguments require analogues of \Cref{thm:basis,thm:bijection} for Pl\"{u}cker degree one; these are established in \Cref{sec:1column}. Section~\ref{sec:combinatorial_consequences} contains combinatorial applications of this work to the Tamari lattice and cyclic sieving.

\section{Background and preliminaries}\label{sec:background}

\subsection{Promotion permutations for standard Young tableaux}
A \emph{partition} with $r$ rows is a tuple $\lambda=(\lambda_1,\ldots,\lambda_r)\in \mathbb{N}_{>0}^r$ where $\lambda_1\geq\cdots\geq\lambda_r$. The \emph{diagram} of $\lambda$ is a collection of upper-left justified boxes where the $i$th row from the top has $\lambda_i$ boxes. A \emph{standard Young tableau (SYT) of  shape} $\lambda$ is a bijective filling of the diagram of $\lambda$ with the numbers $\{1,\ldots,n\}$, where $n=\sum_i\lambda_i$, increasing along rows and columns. 
When $\lambda$ is a rectangle with $r$ rows and $d$ columns, we denote this set as $\SYT(r\times d)$ and we have $n=rd$.

Given $T\in \SYT(r \times d)$, define (following \cite{Schutzenberger-promotion}) the \emph{(jeu de taquin) promotion} of $T$, denoted $\promotion(T)$, as follows. Delete the $1$ from $T$, leaving an empty box. Move the smallest of the  numbers to the right or below the empty box into the empty box. Continue this process with the new empty box, until the empty box is at the lower right corner. Fill the empty box with $n+1$ and then subtract $1$ from all entries. The \emph{promotion path} consists of the numbers that move in this process, and is denoted by arrows in the example of \Cref{fig:prom_orbit}.

In \cite[Def.~6.1]{Gaetz.Pechenik.Pfannerer.Striker.Swanson:fluctuating}, we defined \emph{promotion functions} for \emph{fluctuating tableaux}, a class containing standard Young tableaux. When the shape of the tableau is rectangular, we showed these functions are permutations \cite[Thm.~6.7]{Gaetz.Pechenik.Pfannerer.Striker.Swanson:fluctuating}. In our setting, we do not need the definition in full generality, but give the following, which was \cite[Prop.~6.9]{Gaetz.Pechenik.Pfannerer.Striker.Swanson:fluctuating}.
\begin{definition}
        Let $T \in \SYT(r\times d)$  and $1\leq i\leq r-1$. Then the $i$th \emph{promotion permutation} is constructed as $\prom_i(T)(j) \equiv a+j-1 \pmod n$ if and only if $a$ is the unique value that moves from row $i+1$ to $i$ in the application of jeu de taquin promotion $\promotion$ to  $\promotion^{j-1}(T)$. We write $\prom_{\bullet}(T)$ for the tuple of these promotion permutations. 
\end{definition}

In \cite[Thm.~6.7]{Gaetz.Pechenik.Pfannerer.Striker.Swanson:fluctuating}, we proved the following properties of promotion permutations. Let $\sigma = (1\, 2\, \cdots\, n)$ denote the long cycle and  $w_0 = (1\, n)(2\, n-1)\cdots$ the longest element in the symmetric group $\mathfrak{S}_{n}$, and let $\evacuation(T)$ denote the \emph{evacuation} of $T$ \cite{Schutzenberger-promotion} (see e.g.~\cite[Def.~A1.2.8]{Stanley:EC2}). Then $\prom_i(T)^{-1} = \prom_{r-i}(T)$, $\prom_i(\promotion(T)) = \sigma^{-1} \prom_i(T) \sigma$, and $\prom_i(\evacuation(T)) = w_0 \prom_i(T) w_0$.

\begin{example}\label{ex:proms}
Let $T$ be the first tableau in \Cref{fig:prom_orbit}. We have
\begin{align*}
    \prom_1(T) &= 2\ 6\ 4\ 5\ 9\ 7\ 8\ 12\ 10\ 11\ 14\ 13\ 3\ 1 = \prom_6^{-1}(T), \\
    \prom_2(T) &= 4\ 7\ 5\ 9\ 12\ 8\ 10\ 14\ 11\ 13\ 3\ 1\ 6\ 2  = \prom_5^{-1}(T), \text{ and} \\
    \prom_3(T) &= 5\ 9\ 8\ 12\ 14\ 10\ 11\ 1\ 13\ 3\ 6\ 2\ 7\ 4 = \prom_4^{-1}(T).
\end{align*}
Let $T^\top$ be the transpose of $T$. The standard Catalan bijection sends $T^\top$ to the parenthesization $(()(())())()()$ (see e.g.~\cite[Ex.~6.19(ww), p.~263]{Stanley:EC2}). Matching parentheses yields a non-crossing perfect matching which can be interpreted as a fixed-point free involution, namely the promotion permutation of $T^\top$:
  \[ \prom_1(T^\top) = (1\ 10)(2\ 3)(4\ 7)(5\ 6)(8\ 9)(11\ 12)(13\ 14). \]
\end{example}

\begin{figure}[htbp]
\[T = \begin{tikzpicture}[baseline={([yshift=-.6ex]current bounding box.center)}]
\matrix (m) [matrix of math nodes,
             nodes={draw, minimum size=6.2mm, anchor=center},
             column sep=-\pgflinewidth,
             row sep=-\pgflinewidth
             ]
{
 1 &  3 \\
 2 &  6 \\
 4 &  7 \\
 5 &  9 \\
 8 & 10 \\
11 & 12 \\
13 & 14 \\
};
\def\x{4}
\def\r{7}
\pgfmathtruncatemacro{\y}{\x+1}
\pgfmathtruncatemacro{\rmm}{\r-1}
\ifthenelse{\x>0}{
\foreach \i [evaluate=\i as \j using int(\i+1)] in {1,...,\x} {
\draw[->,shorten >=0.15cm,shorten <=0.15cm,orange,thick] (m-\j-1.center) -- (m-\i-1.center);
};}{}

\draw[->,shorten >=0.15cm,shorten <=0.15cm,orange,thick] (m-\y-2.center) -- (m-\y-1.center);

\ifthenelse{\y<\r}{
\foreach \i [evaluate=\i as \j using int(\i+1)] in {\y,...,\rmm} {
\draw[->,shorten >=0.15cm,shorten <=0.15cm,orange,thick] (m-\j-2.center) -- (m-\i-2.center);
};}{}
\end{tikzpicture}
\quad \stackrel{\promotion}{\to} \quad
\begin{tikzpicture}[baseline={([yshift=-.6ex]current bounding box.center)}]
\matrix (m) [matrix of math nodes,
             nodes={draw, minimum size=6.2mm, anchor=center},
             column sep=-\pgflinewidth,
             row sep=-\pgflinewidth
             ]
{
 1 &  2\\
 3 &  5 \\
 4 &  6 \\
 7 &  8 \\
 9 & 11 \\
10 & 13 \\
12 & 14 \\
};
\def\x{0}
\def\r{7}
\pgfmathtruncatemacro{\y}{\x+1}
\pgfmathtruncatemacro{\rmm}{\r-1}
\ifthenelse{\x>0}{
\foreach \i [evaluate=\i as \j using int(\i+1)] in {1,...,\x} {
\draw[->,shorten >=0.15cm,shorten <=0.15cm,orange,thick] (m-\j-1.center) -- (m-\i-1.center);
};}{}

\draw[->,shorten >=0.15cm,shorten <=0.15cm,orange,thick] (m-\y-2.center) -- (m-\y-1.center);

\ifthenelse{\y<\r}{
\foreach \i [evaluate=\i as \j using int(\i+1)] in {\y,...,\rmm} {
\draw[->,shorten >=0.15cm,shorten <=0.15cm,orange,thick] (m-\j-2.center) -- (m-\i-2.center);
};}{}
\end{tikzpicture}
\quad \stackrel{\promotion}{\to} \quad
\begin{tikzpicture}[baseline={([yshift=-.6ex]current bounding box.center)}]
\matrix (m) [matrix of math nodes,
             nodes={draw, minimum size=6.2mm, anchor=center},
             column sep=-\pgflinewidth,
             row sep=-\pgflinewidth
             ]
{
 1 &  4\\
 2 &  5 \\
 3 &  7 \\
 6 & 10 \\
 8 & 12 \\
 9 & 13 \\
11 & 14 \\
};
\def\x{6}
\def\r{7}
\pgfmathtruncatemacro{\y}{\x+1}
\pgfmathtruncatemacro{\rmm}{\r-1}
\ifthenelse{\x>0}{
\foreach \i [evaluate=\i as \j using int(\i+1)] in {1,...,\x} {
\draw[->,shorten >=0.15cm,shorten <=0.15cm,orange,thick] (m-\j-1.center) -- (m-\i-1.center);
};}{}

\draw[->,shorten >=0.15cm,shorten <=0.15cm,orange,thick] (m-\y-2.center) -- (m-\y-1.center);

\ifthenelse{\y<\r}{
\foreach \i [evaluate=\i as \j using int(\i+1)] in {\y,...,\rmm} {
\draw[->,shorten >=0.15cm,shorten <=0.15cm,orange,thick] (m-\j-2.center) -- (m-\i-2.center);
};}{}
\end{tikzpicture}
\quad \stackrel{\promotion}{\to} \quad
\begin{tikzpicture}[baseline={([yshift=-.6ex]current bounding box.center)}]
\matrix (m) [matrix of math nodes,
             nodes={draw, minimum size=6.2mm, anchor=center},
             column sep=-\pgflinewidth,
             row sep=-\pgflinewidth
             ]
{
 1 &  3\\
 2 &  4 \\
 5 &  6 \\
 7 &  9 \\
 8 & 11 \\
10 & 12 \\
13 & 14 \\
};
\def\x{1}
\def\r{7}
\pgfmathtruncatemacro{\y}{\x+1}
\pgfmathtruncatemacro{\rmm}{\r-1}
\ifthenelse{\x>0}{
\foreach \i [evaluate=\i as \j using int(\i+1)] in {1,...,\x} {
\draw[->,shorten >=0.15cm,shorten <=0.15cm,orange,thick] (m-\j-1.center) -- (m-\i-1.center);
};}{}

\draw[->,shorten >=0.15cm,shorten <=0.15cm,orange,thick] (m-\y-2.center) -- (m-\y-1.center);

\ifthenelse{\y<\r}{
\foreach \i [evaluate=\i as \j using int(\i+1)] in {\y,...,\rmm} {
\draw[->,shorten >=0.15cm,shorten <=0.15cm,orange,thick] (m-\j-2.center) -- (m-\i-2.center);
};}{}
\end{tikzpicture}
\quad \stackrel{\promotion}{\to} \quad \cdots
\]
\caption{A tableau $T\in\SYT(7\times 2)$ and the first several tableaux in its promotion orbit. The promotion path of each tableau is indicated by arrows. Note that a two-column tableau's promotion path always consists of some initial up steps, one left step, and additional up steps. The up steps yield $\prom_i(T)(1)$, e.g.~here $\prom_1(T)(1)=2$, $\prom_2(T)(1)=4$, $\prom_3(T)(1)=5$, $\ldots$, $\prom_6(T)(1)=14$.}
\label{fig:prom_orbit}
\end{figure}
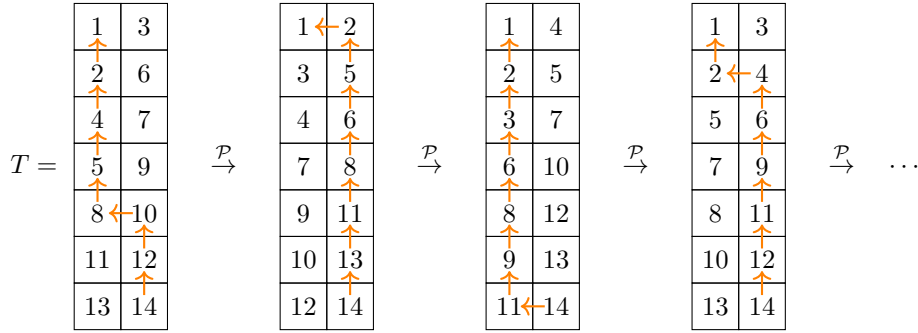

\subsection{Hourglass plabic graphs}
\label{sec:hourglass-background}
Following~\cite[\S3.1]{Gaetz.Pechenik.Pfannerer.Striker.Swanson:4row}, we define hourglass plabic graphs as an extension of Postnikov's plabic graphs \cite{Postnikov-arxiv}. We also define when such graphs are \emph{fully reduced}, generalizing the $r=4$ definition of \cite{Gaetz.Pechenik.Pfannerer.Striker.Swanson:4row}.

An \emph{hourglass graph} $G$ is an underlying planar embedded graph $\widehat{G}$, together with a positive integer multiplicity $m(e)$ for each edge $e$. The hourglass graph $G$ is drawn in the plane by replacing each edge $e$ of $\widehat{G}$  with $m(e)$ strands, twisted so that the clockwise orders of these strands around the two incident vertices are the same. See the lower right portion of \Cref{fig:FraserMap} for an example. For $m(e) \geq 2$, we call this twisted edge an \emph{$m(e)$-hourglass}, and call an edge with $m(e)=1$ a \emph{simple edge}.
The \emph{degree} $\deg(v)$ of a vertex $v \in G$ is the number of edges incident to $v$, counted with multiplicity, while its \emph{simple degree} $\widehat{\deg}(v)$ is its degree in the underlying graph $\widehat{G}$.

\begin{definition}
\label{def:hourglass-plabic-graph}
An \emph{$r$-hourglass plabic graph} is a bipartite hourglass graph $G$, with a fixed proper black--white vertex coloring, embedded in a disk, with all internal vertices of degree $r$, and all boundary vertices of simple degree one, labeled clockwise as $b_1,b_2,\ldots, b_n$. We consider $G$ up to planar isotopy fixing the boundary circle. 

We say that $G$ is of \emph{standard type} if all boundary vertices are colored black and of degree one, and in this case we say that $G$ has \emph{Pl\"ucker degree} $\frac{n}{r}$. (\Cref{lem:white_minus_black} guarantees that the Pl\"ucker degree is an integer.) In the remainder of the paper, all $r$-hourglass plabic graphs are assumed to be of standard type, unless otherwise noted. 
\end{definition}

\begin{definition}\label{def:trip_perms}
Let $G$ be an $r$-hourglass plabic graph with boundary vertices $b_1,\ldots, b_n$. For $1 \leq i \leq r-1$, the \emph{$i$-th trip permutation} $\trip_i(G)$ is the permutation of $[n]$ obtained as follows: for each $j$, begin at $b_{j}$ and walk along the edges of $G$, taking the $i$-th leftmost turn at each white vertex, and $i$-th rightmost turn at each black vertex, until arriving at a boundary vertex $b_{k}$. Then $\trip_i(G)(j)=k$. The walk taken is called a \emph{$\trip_i$-segment} and is denoted $\vec{\trip_i}(G)(j)$. Note that $\trip_i(G)^{-1}=\trip_{r-i}(G)$. 
We write $\trip_{\bullet}(G)$ for the tuple of these trip permutations. 

In addition to the above trip segments which begin and end at the boundary, we also refer to such walks which begin in the interior of the web without reaching the boundary as trip segments. A trip segment is \textit{trivial} if it loops inside a single hourglass edge. An $m$-hourglass has $\binom{m}{2}$ trivial trip segments. In particular, trivial trip segments do not occur on the boundary in standard type. When referring to trip segments below, we always mean non-trivial trip segments unless stated otherwise.
\end{definition}

\begin{definition}
\label{def:intersections}
A trip segment has a \textit{self-intersection} if it passes through a vertex more than once. In particular, all trip segments not reaching the boundary have a self-intersection.

Now suppose $G$ has no self-intersections. We define what it means for two different trip segments $\ell, \ell'$ of $G$ to intersect. First suppose the four endpoints of $\ell, \ell'$ are distinct. Draw the segments $\ell, \ell'$ on the underlying graph $\widehat{G}$. Contract any edges that $\ell, \ell'$ both pass through to a point. The result is a network of $\times$'s which we call \textit{intersections}. Each intersection has four distinct directions at which the two segments $\ell, \ell'$ enter and leave. An intersection is \textit{essential} if $\ell$ and $\ell'$ cross and \textit{inessential} if $\ell$ and $\ell'$ bounce off of each other. If instead segments $\ell$ and  $\ell'$ share a boundary vertex $b_k$, we consider them to have an essential intersection at the unique internal vertex incident to $b_k$. We say $\ell$ and $\ell'$ have an \textit{oriented double crossing} if they have an essential intersection followed in the forwards direction along both segments by another essential intersection.

An \emph{isolated component} of an hourglass graph $G$ is a connected component which does not contain a boundary vertex.
\end{definition}

The following definition is central to the paper.

\begin{definition}
\label{def:fully_reduced_2}
An hourglass plabic graph $G$ is \textit{fully reduced} if it has no isolated components, it has no self-intersections, no two $\trip_i$-segments have an oriented double crossing, and no pair of $\trip_i$- and $\trip_{i+1}$-segments have an oriented double crossing for $1 \leq i \leq r-2$.
\end{definition}

In light of \Cref{def:fully_reduced_2}, we refer to oriented double crossings between a pair of $\trip_i$-segments or between a $\trip_i$- and  a $\trip_{i+1}$-segment as \emph{bad double crossings}.

\begin{remark}
    In \cite{Gaetz.Pechenik.Pfannerer.Striker.Swanson:4row}, we called the condition of \Cref{def:fully_reduced_2} ``monotonic,'' but showed that for $r \leq 4$ it was equivalent to another definition of ``fully reduced.'' 
\end{remark}

\begin{remark}
    For examples of fully reduced hourglass plabic graphs with $\trip_i, \trip_{i+2}$-oriented double crossings, consider $\trip_1(G)(1)$ and $\trip_3(G)(8)$ for either graph $G$ in \Cref{fig:basic-square-move}.
\end{remark}

Given an hourglass plabic graph $G$, the underlying simple graph $\widehat{G}$ is a plabic graph in the sense of Postnikov \cite{Postnikov-arxiv}.
The following fact motivates the terminology of ``fully reduced'' and will be used later. 

\begin{proposition}
\label{prop:underlying-plabic-is-reduced}
Let $G$ be a fully reduced hourglass plabic graph. Then the underlying plabic graph $\widehat{G}$ is reduced.
\end{proposition}
\begin{proof}
For plabic graph terminology used in this proof, we refer the reader to \cite{Postnikov-arxiv,Fomin.Williams.Zelevinsky}.

Note that $\trip_1(G)=\trip(\widehat{G})$ and moreover that all $\trip_1$-segments follow the same sequence of edges and vertices in $G$ and $\widehat{G}$. Since $G$ has no isolated components, $\widehat{G}$ has no isolated components.

We verify the conditions of \cite[Thm.~13.2]{Postnikov-arxiv}: $\widehat{G}$ has no round trips or trips with essential self-intersections because trips in $G$ do not self-intersect. Trips in $\widehat{G}$ have no bad double crossings because the same is true in $G$. Finally, if $\trip(\widehat{G})(j)=j$, then in order to avoid having a self-intersection at the unique internal vertex $v$ incident to $b_j$, it must be that $v$ is a leaf.
\end{proof}

Hourglass plabic graphs admit local transformations called \emph{moves}. The two kinds of moves relevant to this work are the \emph{contraction moves} (see \Cref{fig:contraction-hourglass-moves}) and the \emph{square moves} (see \Cref{fig:square-move}). Graphs differing by a sequence of such moves are \emph{move-equivalent}. An hourglass plabic graph is \emph{contracted} if contraction moves have been applied to convert all subgraphs from the left side of \Cref{fig:contraction-hourglass-moves} to the corresponding graphs from the right side. We will often implicitly follow a square move by the necessary contraction moves to render the resulting graph contracted.

\begin{figure}[htbp]
    \centering
    \includegraphics[scale=1.1]{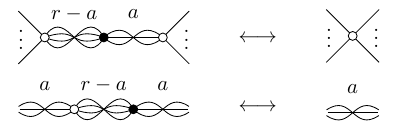}
    \caption{The contraction moves for hourglass plabic graphs. The parameter $a$ runs over $0,\ldots,r$. The color reversals of these moves are also allowed.}
    \label{fig:contraction-hourglass-moves}
\end{figure}

\begin{figure}[hbtp]
    \centering
    \includegraphics[]{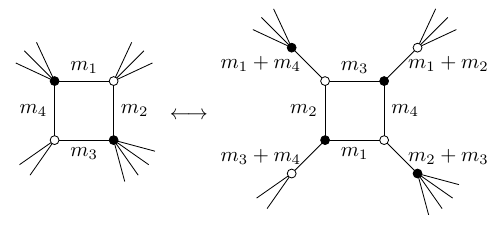}
    \caption{The square move for hourglass plabic graphs. Here $m_1+m_2+m_3+m_4$ is assumed to equal $r$. Multiplicities of the edges on the boundary of the picture are not shown, and are preserved by the move.}
    \label{fig:square-move}
\end{figure}

\subsection{Webs}

An $r$-hourglass plabic graph with boundary vertices $b_1,\ldots,b_n$ corresponds to a \emph{web of type $(\deg(b_1),\ldots,\deg(b_n))$} in which one retains the edge multiplicities but without drawing them as hourglasses. A web $W$ of type $(c_1,\ldots,c_n)$ determines a tensor invariant
\[
[W]_q \in \Inv_{U_q(\mathfrak{sl}_r)}\left(\bigwedge\nolimits_q^{c_1} V_q \otimes \cdots \otimes \bigwedge\nolimits_q^{c_n} V_q \right),
\]
a weighted sum over \emph{proper labelings} of $W$. For a precise definition, refer to \cite{Fraser-Lam-Le} in the $q=1$ case or \cite[\S2.1--2.2]{Gaetz.Pechenik.Pfannerer.Striker.Swanson:4row} in general. Those readers familiar with webs are cautioned that we use the $r$-valent conventions of \cite{Fraser-Lam-Le} as opposed to the trivalent conventions used, e.g., in \cite{Cautis-Kamnitzer-Morrison, Kuperberg, MOY}.

We move freely between the language of webs and of hourglass plabic graphs and, in particular, we refer to the tensor invariant of an hourglass plabic graph.

\subsection{Fraser's map}
\label{sec:Fraser_map}
We now describe the map of Fraser \cite[\S 2.1--2.3]{Fraser-2-column} from $2$-column standard Young tableaux to certain webs, which we view as hourglass plabic graphs. See \Cref{fig:FraserMap} for an example.
Define the map $\mathcal{F}$ on $\SYT(r\times 2)$ as the composition of maps $\mathcal{F}(T) \coloneqq \mathscr{W}\circ\mathfrak{t}\circ\dis\circ \mathscr{M}(T)$ where
\begin{align*}
\SYT(r\times 2) &\xrightarrow{\mathscr{M}} \{\text{noncrossing matchings on } 2r \text{ points}\}\\
&\xrightarrow{\dis}\displaystyle\bigcup_{s=2}^{r}\{\text{weighted dissections of an } s\text{-gon}\}\\
&\xrightarrow{\mathfrak{t}}\displaystyle\bigcup_{s=2}^{r}\{\text{weighted triangulations of an } s\text{-gon}\}\\
&\xrightarrow{\mathscr{W}} \{\text{webs of type $(1^{2r})$}\}\\
&\cong \{r\text{-hourglass plabic graphs of Pl\"ucker degree }2\}.
\end{align*}
Note that this map is not a bijection. The map $\mathfrak{t}$ involves choices, so this map actually takes a standard Young tableau to an equivalence class of weighted triangulations, and thus to an equivalence class of hourglass plabic graphs.

\medskip
\noindent{\textbf{\emph{Construction of the matching $\mathscr{M}$:}}}

First, draw the matching given by the standard Catalan bijection of the transposed tableau. That is, given $T\in\SYT(r\times 2)$, let $\mathscr{M}(T)$ be the unique noncrossing perfect matching in which each number in the first column is the smallest in its pair (see e.g. \cite[p.~90]{Stanley-catalan}). 

\medskip
\noindent{\textbf{\emph{Construction of the weighted dissection  $\dis$:}}}

Given a noncrossing matching $M$ of $2r$ points with $s$ short arcs (arcs between adjacent boundary vertices), let $i_1,i_2,\ldots,i_s$ be the sequence of numbers such that $i_j$ and $i_j+1$ mod $2r$ are joined by an arc. That is, $i_1,i_2,\ldots$ are the cyclically left endpoints of all the short arcs. The \emph{claw sets} are the cyclic intervals $C_j\coloneqq (i_j,i_{j+1}]$. (These were called \emph{color sets} in \cite{Fraser-2-column}.)
We use the claw sets to obtain a \emph{weighted dissection} $\dis(M)$ of the $s$-gon by merging the boundary vertices in each claw set and replacing multiple edges by a single edge with the corresponding weight, i.e.~the number of matchings between the two claw sets.
Note that the total weight of $\dis(M)$ is the number of arcs in $M$, which is $r$. Similarly,
the weight of $\dis(M)$ at boundary vertex $j$ in the $s$-gon equals the cardinality of the claw set $C_j$ of $M$.

\medskip
\noindent{\textbf{\emph{Construction of a weighted triangulation $\mathfrak{t}$:}}}

Given such a weighted dissection $\dis$, we then add in diagonals of weight zero to create a weighted triangulation $\mathfrak{t}(\mathfrak{d})$. (This step is not unique; any choice of triangulation is related to any other by a sequence of flip moves on weight zero diagonals.)

\medskip
\noindent{\textbf{\emph{Construction of the web} $\mathscr{W}$:}}

Given such a weighted triangulation $\mathfrak{t}$, construct a web $\mathscr{W}(\mathfrak{t})$ as follows. $\mathscr{W}(\mathfrak{t})$ has $2r$ black boundary vertices and $s$ internal white vertices, one for each vertex of the $s$-gon. The $j$th white vertex in $\mathscr{W}(\mathfrak{t})$ is adjacent to the entire claw set $C_j$ of boundary vertices by weight $1$ edges.  $\mathscr{W}(\mathfrak{t})$ also has an internal trivalent black vertex $b(\Delta)$ for each triangle $\Delta$ in $\mathfrak{t}$.  Let $\delta_1,\delta_2,\delta_3$ denote the vertices of $\Delta$ in $\mathfrak{t}$ and $w(\delta_1),w(\delta_2),w(\delta_3)$ denote the corresponding white vertices in $\mathscr{W}(\mathfrak{t})$; these are the three vertices adjacent to $b(\Delta)$  in $\mathscr{W}(\mathfrak{t})$. The weights of these three edges are given as follows. Let $e_i$ denote the edge of $\Delta$  opposite to $\delta_i$. The edge $e_i$ cuts the triangulation $\mathfrak{t}$ into two parts; call (*) the one that does not contain $\Delta$. Then the weight of the edge between $b(\Delta)$ and $w(\delta_i)$ is the sum of the weights of the edges in (*), including the weight of the edge $e_i$.

We append a final step to the construction of \cite{Fraser-2-column}. Given the web $\mathscr{W}$, replace each edge of weight $m>1$ with an hourglass of multiplicity $m$ in order to form an $r$-hourglass plabic graph $\mathcal{F}(T)$. See \Cref{fig:FraserMap,fig:triflip}.

We say that an hourglass plabic graph is \emph{Fraser} if it can be produced from this construction. Fraser showed that the result of this construction is a web, and that the web invariant does not depend on the choice of triangulation. We rephrase this in the proposition below. If $G$ is a Fraser graph arising from the tableau $T$, we write $\mathcal{T}(G) = T$.

\begin{figure}[htb]
\includegraphics[width=0.3\textwidth]{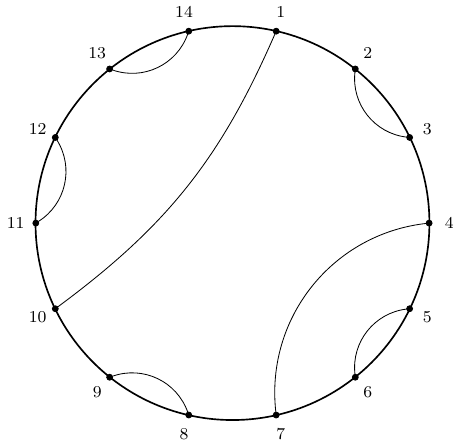}
\includegraphics[width=0.3\textwidth]{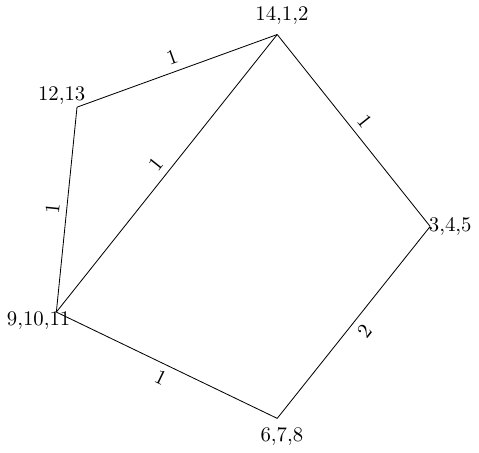}\includegraphics[width=0.3\textwidth]{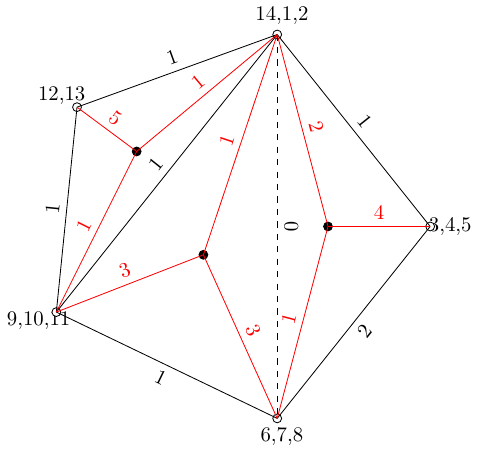}

\includegraphics[width=0.4\textwidth]{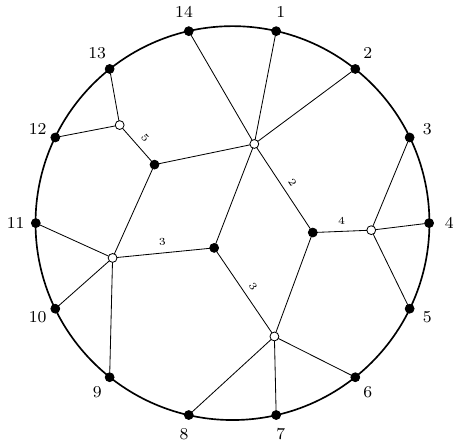} \includegraphics[width=0.4\textwidth]{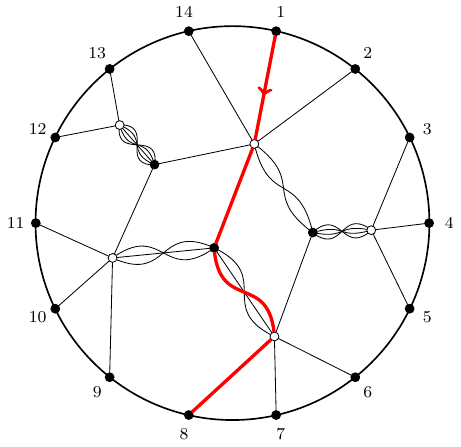}
\caption{An illustration of the construction discussed in \Cref{sec:Fraser_map}. Top left: The noncrossing matching corresponding to the transpose of the first SYT of \Cref{fig:prom_orbit}. Top middle and right: The corresponding weighted dissection $\mathfrak{d}$ and one choice of weighted triangulation~$\mathfrak{t}$, with the interior of the web from the next step drawn in red. Bottom left: A web produced from $\mathscr{W}$ via the construction of Fraser \cite{Fraser-2-column}. Bottom right: The hourglass plabic graph $G$ constructed by replacing edges  with hourglasses. The highlighted trip segment computes  $\trip_4(G)(1)=8$, in agreement with \Cref{ex:proms} which gave $\prom_4(T)(1)=8$.}
\label{fig:FraserMap}
\end{figure}
\begin{proposition}[cf.\ {\cite[Prop.~2.14]{Fraser-2-column}}]
\label{prop:Fraser_web_invariant}
For a fixed $T \in \SYT(r \times 2)$, the graph $\mathcal{F}(T)$ is determined uniquely, up to square moves. Moreover, the web invariant of $\mathcal{F}(T)$ does not depend on the choices involved.
\end{proposition}

\begin{figure}[htbp]
\includegraphics[width=0.4\textwidth]{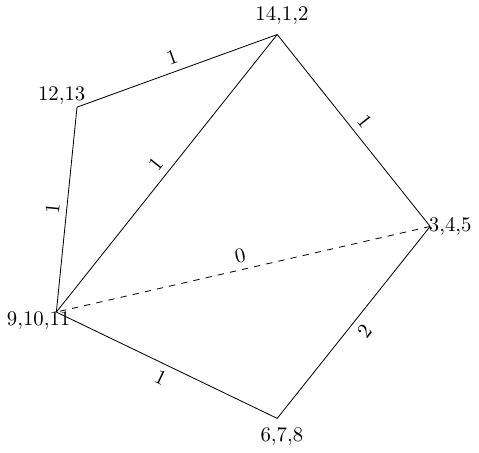}
\includegraphics[width=0.4\textwidth]{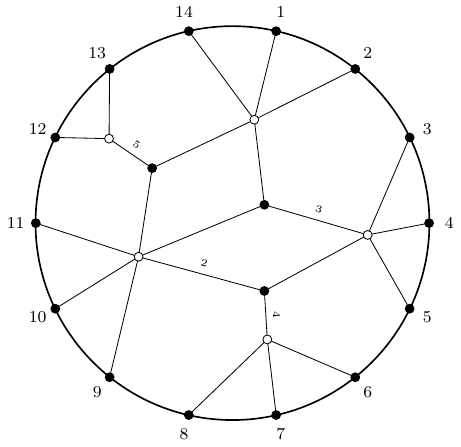} 
\caption{A graph related to the graph of \Cref{fig:FraserMap} by a  flip of a weight $0$ diagonal in the triangulation, and equivalently, a square move.}
\label{fig:triflip}
\end{figure}

The smallest case where such equivalence classes are absolutely necessary is in $\Inv_{\SL_4}(V^{\otimes 8})$, which is the basic $r=4$ square move living in Pl\"ucker degree $2$ (see \Cref{fig:basic-square-move}). Here the corresponding standard Young tableau has promotion order 2, while either of these graphs has rotation order 4. But the entire equivalence class has rotation order $2$. As we shall see, move equivalence classes have rich connections to statistical mechanics and lattices. In our experience, what may be initially perceived as a defect is in fact a feature of the framework.

\begin{figure}[htb]
    \ytableausetup{centertableaux}
    \centering
    {\includegraphics[width=0.7\textwidth]{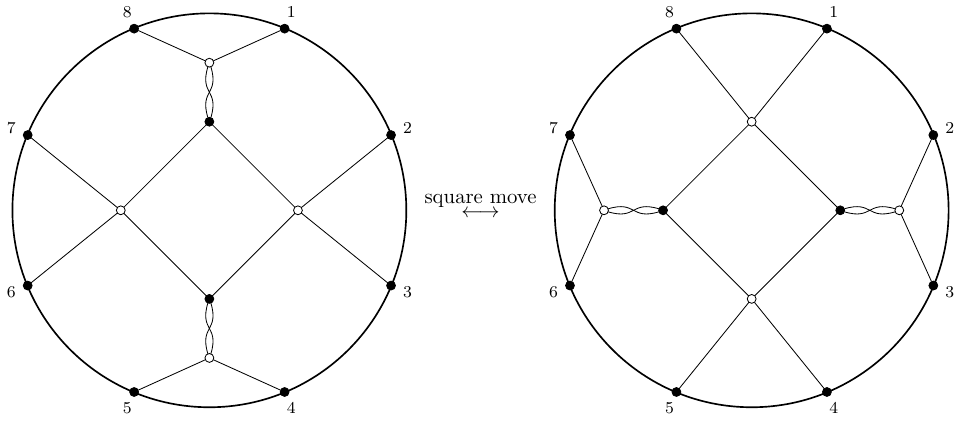}}
    \caption{The square move equivalence class $\mathcal{F}(T)$ for $T={\tiny \ytableaushort{12,34,56,78}}$.}
    \label{fig:basic-square-move}
\end{figure}

\section{The \texorpdfstring{$\trip_{\bullet}=\prom_{\bullet}$ property}{trip=prom}}\label{sec:trip-prom}
The main result of this section is the following.
\begin{theorem}\label{thm:trip=prom}
The map $\mathcal{F}$ from $\SYT(r \times 2)$ to Fraser graphs satisfies $\trip_\bullet(\mathcal{F}(T)) = \prom_\bullet(T)$.
\end{theorem}

Our approach identifies ``ears'' in Fraser graphs which may be removed inductively. We show that the equality $\trip_\bullet = \prom_\bullet$ is appropriately preserved under ear removal.

\subsection{Promotion permutations for two column rectangles}
We regard noncrossing perfect matchings as permutations $[2r] \to [2r]$. For such a matching $M$, we say that $i$ is an \newword{opener} if $i<M(i)$ and $i$ is a \newword{closer} otherwise.

Let $T \in \SYT(r \times 2)$ and let $\mathscr{M}(T):[2r] \to [2r]$ be the corresponding noncrossing perfect matching. We call the arc $\{1, \mathscr{M}(T)(1)\}$ the \emph{barrier}. Let $o_1 < \cdots < o_k$ be the openers of $\mathscr{M}(T)$ strictly between $1$ and $\mathscr{M}(T)(1)$ and let $c_{k+1} < \cdots < c_{r-1}$ be the closers of $\mathscr{M}(T)$ strictly larger than $\mathscr{M}(T)(1)$. Note that every arc in $\mathscr{M}(T)$ except the barrier has exactly one element in $\{o_1, \ldots, o_k, c_{k+1}, \ldots, c_{r-1}\}$.

\begin{proposition}\label{prop:prom1}
With the above notation, we have
\[
\prom_{i}(T)(1) = \begin{cases}
o_i&\text{if $1 \leq i\le k$,}\\
c_i&\text{if $k+1 \leq i \leq r-1$.}
\end{cases}
\]
\end{proposition}
\begin{proof}
    The openers of $\mathscr{M}(T)$ are in the first column of $T$, and the closers are in the second column of $T$. Applying promotion to $T$ results in the following sliding procedure (after deleting $1$):
    \begin{enumerate}
        \item The top $k$ entries from the left column are moved up;
        \item the entry $e$ is moved from the right column to the left column in row $k+1$; and
        \item the entries below $e$ in the right column are moved up.
    \end{enumerate}

    Note that $e=\mathscr{M}(T)(1)$, the entries moved in the first part of the process are precisely the openers $o_1, \dots, o_k$, and the entries in the last part are precisely the closers $c_{k+1}, \dots, c_{r-1}$.
\end{proof}

\subsection{Claws and ears}
In this subsection, we define notation and prove lemmas that will be used in the next subsection to prove \Cref{thm:trip=prom}.
\begin{definition}
    Let $G$ be an $r$-hourglass plabic graph.
 A \textit{claw} $C$ of $G$ is a consecutive collection of boundary vertices connected to the same white internal vertex by edges with multiplicity~$1$. The \textit{size} of a claw is the number of its boundary vertices. The common internal vertex is the \textit{center} of the claw.

        An \textit{ear} $(A, B, C)$ of $G$ is a particular configuration of claws $A, B, C$ as in \Cref{fig:claws-ears}. Here $A, B, C$ are in clockwise order with centers $a, b, c$, respectively. We require that $a, b, c$ be connected to a black internal vertex $v$ of simple degree $3$ and that $b$ has no edges other than those depicted. The \textit{type} of $(A, B, C)$ is $(p, q)$ where the multiplicity of the hourglass between $a$ and $v$ is $q$ and the multiplicity of the hourglass between $c$ and $v$ is $p$. We furthermore require the size of $A$ to be at least $p$ and the size of $C$ to be at least $q$.
   
        Given an ear $(A, B, C)$, the clockwise-last $p$ boundary edges of $A$ are the \textit{lobe of $A$}, the clockwise-first $q$ boundary edges of $C$ are the \textit{lobe of $C$}, the clockwise-first $p$ boundary edges of $B$ are the \textit{flap of $A$}, and the clockwise-last $q$ boundary edges of $B$ are the \textit{flap of $C$}. The three \textit{ossicles} of the ear are the hourglass edges connecting $a, b, c$ to $v$. We refer to these as the $A$-, $B$-, and $C$-ossicles.

         Call an ear of a Fraser graph \textit{proper} if neither vertex of the barrier $\{1, \mathscr{M}(T)(1)\}$ appears as a boundary vertex of the claw $B$. 
\end{definition}

Since all boundary vertices have degree $1$, we often conflate boundary vertices with their incident edges in the context of flaps, lobes, claws, etc.

\begin{figure}[htbp]
    \centering

    \raisebox{-0.5\height}{\includegraphics[]{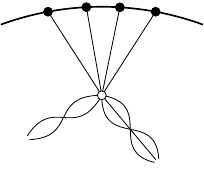}}
    \hspace*{1.4cm}
    \raisebox{-0.5\height}{\includegraphics[scale=0.8]{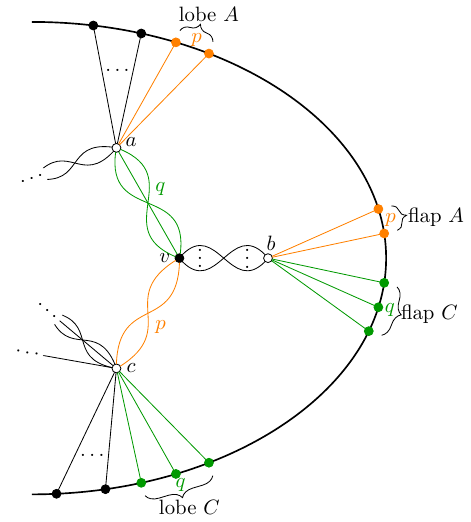}} 
    \caption{Left: A claw of size $4$. Right: An ear $(A, B, C)$ of type $(p, q)$ with its lobes and flaps.}
    \label{fig:claws-ears}
\end{figure}

In an ear $(A, B, C)$ of type $(p, q)$, note that the size of $B$ is $p+q$. If $G$ is Fraser, then by construction the lobe and flap of $A$ are $p$ matched pairs of $\mathscr{M}(T)$, and the lobe and flap of $C$ are $q$ matched pairs, which account for all $p+q$ boundary vertices of $B$. If additionally the ear $(A,B,C)$ is proper, then the lobes of $A$ and $C$ are necessarily also disjoint from the barrier.

\begin{definition}
    To \textit{remove an ear} $(A,B,C)$ of type $(p,q)$ in an hourglass plabic graph $G$, do the following; see \Cref{fig:ear-removal} for a visual illustration. First delete the claw $B$ as well as the vertex $v$ and the ossicles. Add $q$ clockwise-last boundary edges to $A$, resulting in a new claw $A'$, and add $p$ clockwise-first boundary edges to $C$, resulting in a new claw $C'$. The \textit{lobe of $A'$} is the lobe of $A$, the \textit{lobe of $C'$} is the lobe of $C$, the \textit{flap of $C'$} consists of the $q$ added boundary edges of $A'$, and the \textit{flap of $A'$} consists of the $p$ added boundary edges of $C'$.
\end{definition}

\begin{figure}[htp]
\[
\includegraphics[width=0.4\textwidth]{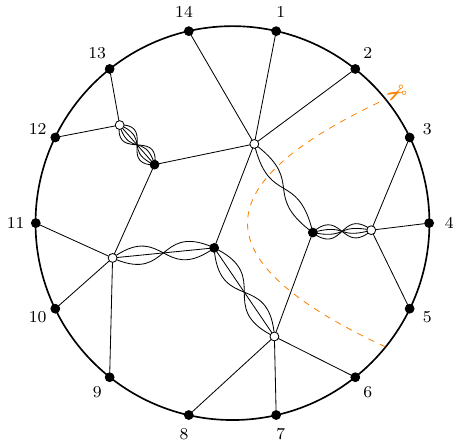} 
\includegraphics[width=0.4\textwidth]{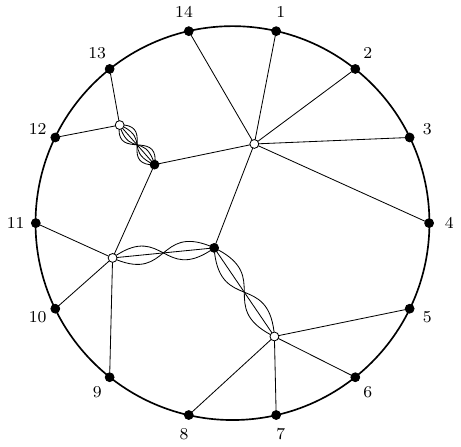}
\]
    \centering
    \includegraphics[width=.8\textwidth]{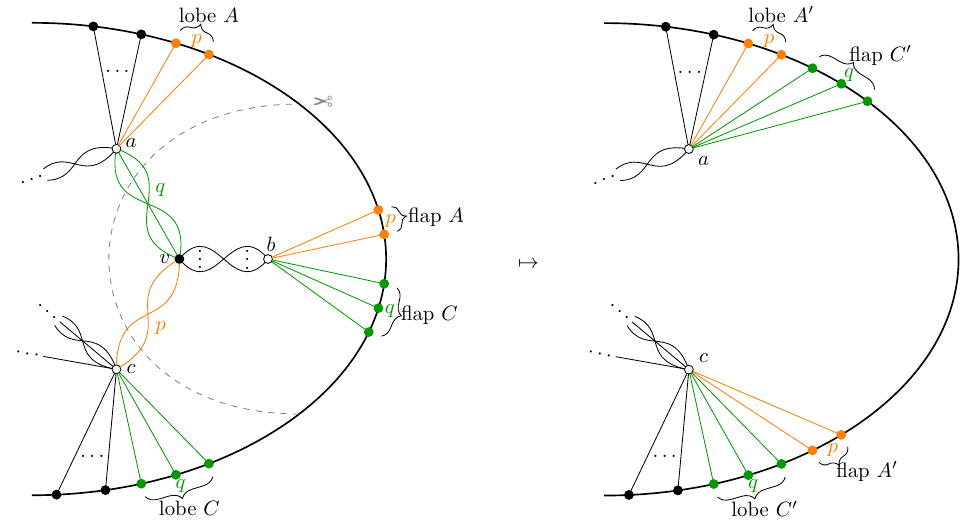}
    \caption{Removal of an ear $(A, B, C)$ of type $(p, q)$ resulting in new claws $(A', C')$.}
    \label{fig:ear-removal}
\end{figure}

Suppose $G$ is Fraser, arising from the triangulation $\mathfrak{t}$. Ears of $G$ correspond precisely to triangles in $\mathfrak{t}$ with at least two edges on the boundary of the polygon. Removing an ear corresponds precisely to collapsing such a triangle as in \Cref{fig:collapse-triangle}. 
Hence we have the following.

\begin{lemma}\label{lem:ear_removal}
    Removing an ear from a Fraser graph results in a Fraser graph with fewer vertices. \qed
\end{lemma}

\begin{figure}[htbp]
    \centering
    \includegraphics[width=.8\textwidth]{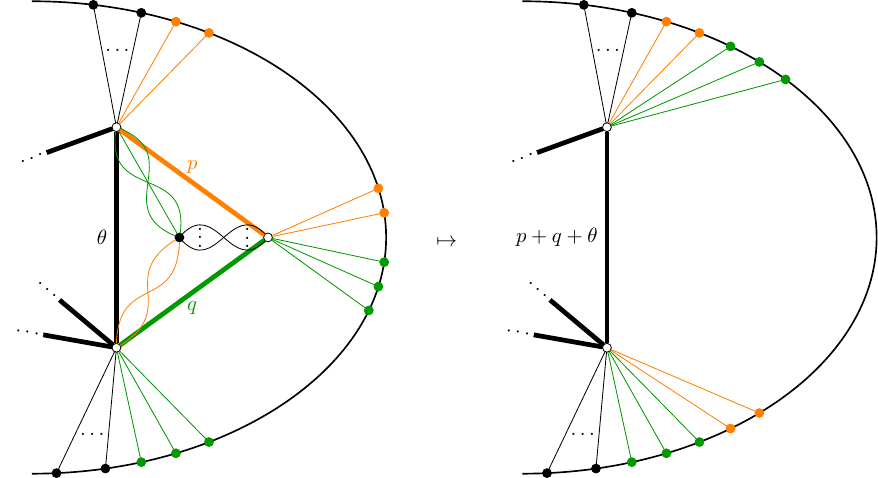}
    \caption{Ear removal of a Fraser graph corresponds to collapsing a triangle of the triangulation.}
    \label{fig:collapse-triangle}
\end{figure}

\begin{lemma}\label{lem:proper-ears}
    Every Fraser graph $G$ has a proper ear, except for the disjoint union of two claws.
\end{lemma}

\begin{proof}
    Let $\mathfrak{t}$ be the triangulation corresponding to $G$. Either $\mathfrak{t}$ is the vacuous triangulation of a $2$-gon, $\mathfrak{t}$ is a triangle, or $\mathfrak{t}$ has at least two triangles with two edges on the boundary, which hence correspond to ears. In the first case, $G$ is the disjoint union of two claws. In the second case, we may cycle the three claws $(A, B, C)$ so that $1$ and $\mathscr{M}(T)(1)$ are not in $B$. In the third case, we in fact have two ears $(A, B, C)$ and $(D, E, F)$ where $E \notin \{A, B, C\}$. If $(A, B, C)$ is not proper, then $B$ intersects the barrier, and since every boundary vertex of $B$ is matched with lobes in $A$ or $C$, the barrier is entirely within $(A, B, C)$. Hence the barrier cannot intersect $E$, so $(D, E, F)$ is proper.
\end{proof}

\subsection{Proof of \texorpdfstring{\Cref{thm:trip=prom}}{Theorem 3.1}}

We first require the following simple observation whose proof is immediate.

\begin{claim}[Tunneling through an hourglass]\label{lem:tunnel}
    Suppose a $\trip_i$-segment enters an hourglass edge from the counterclockwise-$j$th strand after the hourglass. Then it exits the hourglass edge on the counterclockwise-$j$th strand after the hourglass; see \Cref{fig:tunnel}. \qed 
\end{claim}

\begin{figure}[htp]
    \centering
    \includegraphics[]{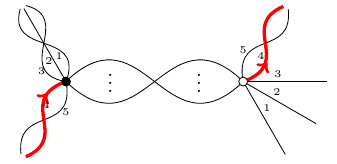}
    \caption{Tunneling through an hourglass. Here we assume the highlighted trip segment passes through the middle hourglass.}
    \label{fig:tunnel}
\end{figure}

    We now prove \Cref{thm:trip=prom} by induction on the number of vertices of $G$. In the base case, $G$ is the disjoint union of two claws, and the result is straightforward to verify. Otherwise, by \Cref{lem:proper-ears}, $G$ has a proper ear $(A, B, C)$. Let $H$ be the result of removing this ear and let $(A', C')$ be the resulting claws. By \Cref{lem:ear_removal}, $H$ is Fraser with fewer vertices than $G$, so by induction $\trip_\bullet(H) = \prom_\bullet(\mathcal{T}(H))$.

We now use this to show $G$  satisfies $\trip_\bullet(G) = \prom_\bullet(\mathcal{T}(G))$.
    We begin by showing $\trip_i(G)(1) = \prom_i(\mathcal{T}(G))(1)$ when $A, B, C$ are clockwise-between $1$ and the other endpoint of the barrier of $G$. In this case, the lobe of $A$ consists of openers connected to the flap of $A$, and the flap of $C$ consists of openers connected to the lobe of $C$. Furthermore, the lobe of $A'$ and the flap of $C'$ consist of openers which are matched with the flap of $A'$ and the lobe of $C'$. See \Cref{fig:ear-removal}.

    Consider $\trip_i(G)(1)$ and $\prom_i(\mathcal{T}(G))(1)$. By \Cref{prop:prom1}, $\prom_\bullet(\mathcal{T}(G))(1)$ are the $k$ openers to the right of the barrier, including the openers in $(A, B, C)$, along with the $r-1-k$ closers to the left of the barrier. The same holds for $\prom_\bullet(\mathcal{T}(H))(1)$. By inductive hypothesis, $\trip_\bullet(H)(1) = \prom_\bullet(\mathcal{T}(H))(1)$. 
    
First consider an opener or closer $u$ of $G$ not in the claw $B$. Note that $u$ is also an opener or closer of $H$ not in the flap of $C'$ or the flap of $A'$. In particular, by \Cref{prop:prom1},
        \[
        \trip_i(H)(1) = \prom_i(\mathcal{T}(H))(1) = u = \prom_i(\mathcal{T}(G))(1)
        \]
        for some $i$. Since the segment $\vec{\trip}_i(H)(1)$ does not end in the flap of $C'$ or the flap of $A'$, $\vec{\trip}_i(G)(1)$ does not enter the $A$- or $C$-ossicle, so that $\trip_i(G)(1) = \trip_i(H)(1)$. Thus, for such $i$, $\trip_i(G)(1) = \prom_i(\mathcal{T}(G))(1)$.
        
        Now consider an opener or closer $u$ of $G$ in the claw $B$, namely $u$ is an opener in the flap of $C$. We again have $u = \prom_i(\mathcal{T}(G))(1)$ for some $i$. Say $u$ is the counterclockwise-$j$th edge in the flap of $C$. Correspondingly, if $u'$ is the counterclockwise-$j$th edge in the flap of $C'$, then we see $\trip_i(H)(1) = \prom_i(\mathcal{T}(H))(1) = u'$. The segment $\vec{\trip}_i(H)(1)$ enters the flap of $C'$ and not the flap of $A'$, so the segment $\vec{\trip}_i(G)(1)$ enters the $A$-ossicle on the counterclockwise-$j$th strand. It then enters $v$ on the counterclockwise-$j$th strand. If it enters the $B$-ossicle, then by \Cref{lem:tunnel} it exits on the counterclockwise-$j$th strand of the flap of $C$, in which case $\trip_i(G)(1) = u$. Since it takes the $i$th right at $v$, we must show $p+q < i+j$. Note that $i$ is the number of openers between $1$ and $u$, which includes the $p$ vertices in the lobe of $A$ and the $q-j+1$ clockwise-first vertices in the flap of $C$, so indeed $p+q < i+j$.

    If instead $A, B, C$ are to the left of the barrier of $G$, we may reduce to the above case by reflecting through a diameter passing between $1$ and $n$, since the matching is noncrossing. This operation conjugates $\trip_\bullet(G)$ by the long element $w_0$. It reflects the underlying triangulation $\mathfrak{t}$ of $G$, which corresponds to evacuation on $\mathcal{T}(G)$. Evacuation conjugates $\prom_\bullet(\mathcal{T}(G))$ by the long element $w_0$ \cite[Thm.\ 6.7]{Gaetz.Pechenik.Pfannerer.Striker.Swanson:fluctuating}. Thus $\prom_\bullet(\mathcal{T}(G))(1) = \trip_\bullet(G)(1)$ unconditionally. Similarly, by \cite[Thm.\ 6.7]{Gaetz.Pechenik.Pfannerer.Striker.Swanson:fluctuating}, rotation corresponds to conjugation by the long cycle $\sigma$, so $\prom_\bullet(\mathcal{T}(G)) = \trip_\bullet(G)$. \qed

\section{Fraser graphs are fully reduced}
\label{sec:fraser_is_fr}

We now extend the arguments in \Cref{sec:trip-prom} to show that Fraser graphs are fully reduced (and contracted) in the sense of \Cref{def:fully_reduced_2}. Throughout this section, $G$ is a Fraser graph with ear $(A, B, C)$, $H$ is the Fraser graph obtained by removing this ear, and $(A', C')$ are the resulting claws in $H$. We begin with some observations concerning trip segments. 

\begin{lemma}\label{lem:trips-ossicles}
     Let $\ell = \vec{\trip}_i(G)(j)$ be some trip segment which passes through an ossicle. Then $\ell$ or its reversal has one of the following forms:
    \begin{enumerate}[(i)]
        \item begin in the lobe of $A$, pass through the $A$-ossicle, pass through the $C$-ossicle, and terminate in the lobe of $C$; or
        \item begin in the lobe of $C$, pass through the $C$-ossicle, pass through the $A$-ossicle, and terminate in the lobe of $A$; or
        \item begin in the flap of $C$, pass through the $B$-ossicle, pass through the $A$-ossicle, and continue into $G$, not into the lobe of $A$; or
        \item begin in the flap of $A$, pass through the $B$-ossicle, pass through the $C$-ossicle, and continue into $G$, not into the lobe of $C$.
    \end{enumerate}
\end{lemma}

\begin{proof}
    Suppose $\ell$ passes through the $A$-ossicle. By \Cref{lem:tunnel}, it must pass through either the $B$- or $C$-ossicle. If it passes through the $C$-ossicle, then \Cref{lem:tunnel} forces it to pass from the lobe of $A$ to the lobe of $C$ as in case (i). If it passes through the $B$-ossicle, then \Cref{lem:tunnel} forces it to pass through the flap of $C$ on one end and to continue into $G$ on the other, not into the lobe of $A$, as in case (iii).

    Symmetrically, if $\ell$ passes through the $C$-ossicle, we have cases (ii) or (iv). Finally, if $\ell$ passes through the $B$-ossicle, then \Cref{lem:tunnel} forces it to pass through the $A$- or $C$-ossicle, which has already been handled.
\end{proof}

\begin{corollary}\label{cor:not-in-lobes-flaps}
    If $j$ and $\trip_i(G)(j)$ are not in the lobes or flaps of the ear, then \[\vec{\trip}_i(G)(j) = \vec{\trip}_i(H)(j).
    \]
\end{corollary}

\begin{proof}
    By \Cref{lem:trips-ossicles}, $\vec{\trip}_i(G)(j)$ cannot pass through the ossicles. Hence it is entirely confined to the portions of $G$ which are identical in $H$.
\end{proof}

We will also require the following more precise version of \Cref{lem:trips-ossicles} in the case of Fraser graphs.

\begin{lemma}\label{lem:trips-G-H}
    Let $G$ be Fraser. In the cases of \Cref{lem:trips-ossicles}, we have the following corresponding trip segments $\ell'$ in $H$.
    \begin{enumerate}[(i)]
        \item $\ell'$ begins in the lobe of $A'$ at the same relative place as $\ell$ begins in the lobe of $A$, and then $\ell'$ immediately ends in the flap of $C'$ at the same relative place as $\ell$ ends in the lobe of $C$.
        \item $\ell'$ begins in the lobe of $C'$ at the same relative place as $\ell$ begins in the lobe of $C$, and then $\ell'$ immediately ends in the flap of $A'$ at the same relative place as $\ell$ ends in the lobe of $A$.
        \item $\ell'$ begins in the flap of $C'$ at the same relative place as $\ell$ begins in the flap of $C$, and then $\ell'$ continues exactly as $\ell$ continues once it leaves the $A$-ossicle.
        \item $\ell'$ begins in the flap of $A'$ at the same relative place as $\ell$ begins in the flap of $A$, and then $\ell'$ continues exactly as $\ell$ continues once it leaves the $C$-ossicle.
    \end{enumerate}
    Moreover,
    \begin{enumerate}[(a)]
        \item there is a unique trip between any two distinct vertices of the lobes of $A$ and $C$, which either falls in cases (i) or (ii) or stays entirely within a single lobe, and
        \item in cases (iii) and (iv), $\ell$ does not re-enter the ear after having left.
    \end{enumerate}
\end{lemma}

\begin{proof}
    The proof of (i)--(iv) only requires examining the argument for \Cref{lem:trips-ossicles} a little more carefully. For (b), suppose $\ell$ starts at some vertex $u$ in the claw $B$, and suppose to the contrary $\ell$ later re-enters the ear. The reversal of $\ell$ must belong to case (iii) or (iv), i.e.~$\ell$ ends at some vertex $w$ of the claw $B$. By \Cref{thm:trip=prom}, $\trip_\bullet(G)(u) = \prom_\bullet(T(G))(u)$. By \Cref{prop:prom1}, the collection of endpoints $\trip_\bullet(G)(u)$ are distinct and do not include $u$. Thus $w \neq u$ and there is at most one trip from $u$ to $w$. However, there is a trip segment directly from $u$ to $w$ which does not enter the $B$-ossicle, a contradiction. The argument for (a) is very similar.
\end{proof}

\begin{corollary}\label{cor:no_self_intersection}
    Fraser graphs have no self-intersections.
\end{corollary}
\begin{proof}
    We induct on the number of vertices of a Fraser graph. In the base case, $G$ is a disjoint union of claws, and the result is straightforward. Otherwise, we have $G$ and $H$ as above, where by induction $H$ has no self-intersections.

    First suppose $\ell$ is a trip segment in $G$ which does not reach the boundary. By \Cref{lem:trips-ossicles}, $\ell$ cannot travel through an ossicle, so it is entirely contained in the portions of $G$ which are identical in $H$, contradicting the fact that trip segments in $H$ reach the boundary.

    Now suppose $\ell$ is a trip segment in $G$ which reaches the boundary. If $\ell$ does not pass through an ossicle, then it is either entirely confined to the portions of $G$ which are identical in $H$, or it is confined to the claw of $B$. In the former case, $\ell$ is the same in $H$ and $G$ and hence has no self-intersections. In the latter case, $\ell$ is trivial and does not have a self-intersection.

    Finally, suppose $\ell$ passes through an ossicle, so \Cref{lem:trips-ossicles} applies. \Cref{lem:trips-ossicles}(i)--(ii) do not involve self-intersections. \Cref{lem:trips-ossicles}(iii)--(iv) likewise do not involve self-intersections since $\ell'$ in \Cref{lem:trips-G-H}(iii)--(iv) does not re-enter the ear after having left and by induction $\ell'$ does not have a self-intersection.
\end{proof}

We will need the following topological result concerning fully reduced hourglass plabic graphs.

\begin{lemma}\label{lem:fully-reduced-crossings}
    Let $G$ be a fully reduced hourglass plabic graph. Suppose $\ell_1$ is a $\trip_i$-segment from $a$ to $b$ and $\ell_2$ is a $\trip_j$-segment from $c$ to $d$, where $a, b, c, d$ are distinct, $a$ is their minimum, $c < b$, and either $j=i$ or $j = i\pm1$. Then
    \begin{enumerate}[(i)]
        \item $a < c < d < b$ and $\ell_1$, $\ell_2$ have no essential intersections; or
        \item $a < d < c < b$ and $\ell_1$, $\ell_2$ have an even number of essential intersections; or
        \item $a < c < b < d$ and $\ell_1$, $\ell_2$ have an odd number of essential intersections.
    \end{enumerate}
\end{lemma}

\begin{proof}
    Draw the trip segments and their essential intersections. Using the fact that $\ell_1$ and $\ell_2$ do not have self-intersections or a bad double crossing results in only the configurations of \Cref{fig:topological}.
\end{proof}

\begin{figure}[htp]
    \centering
    \includegraphics[width=0.3\textwidth]{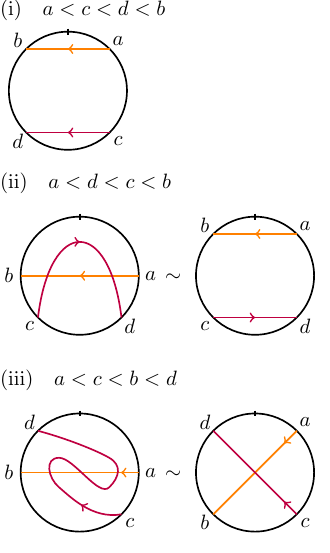}
    \caption{Possible configurations of trip segments in \Cref{lem:fully-reduced-crossings}.}
    \label{fig:topological}
\end{figure}

We are now ready to prove the main result of this section.

\begin{theorem}\label{thm:fraser-implies-fully-reduced}
Fraser graphs are fully reduced. 
\end{theorem}

\begin{proof}
    We proceed by induction on the number of vertices of a Fraser graph. In the base case, $G$ is a disjoint union of claws, and the result is straightforward. Otherwise we have $G$ and $H$ as above, where $H$ is Fraser, hence fully reduced by the inductive hypothesis, and $G$ has no self-intersections by \Cref{cor:no_self_intersection}.

    Suppose $\ell_1$ is a $\trip_i$-segment of $G$ and $\ell_2$ is a distinct $\trip_i$- or $\trip_{i \pm 1}$-segment of $G$. If $\ell_1$ and $\ell_2$ do not begin or end at the lobes or flaps, then they are trip segments of $H$ by \Cref{cor:not-in-lobes-flaps}, and hence do not have bad double crossings by the inductive hypothesis. Similarly, if $\ell_1$ does not begin or end at the lobes or flaps and $\ell_2$ does, by \Cref{lem:trips-G-H} we see that $\ell_1$ and $\ell_2$ do not have bad double crossings. In this way, we may consider the case when $\ell_1$ and $\ell_2$ both either begin or end in the lobes or flaps.

    If $\ell_1$ and $\ell_2$ begin and end in the lobes, then from \Cref{lem:trips-G-H}(i)--(ii) they have at most one essential crossing. If $\ell_1$ begins and ends in the lobes and $\ell_2$ begins or ends in the flaps, then from \Cref{lem:trips-G-H}(i)--(ii) and \Cref{lem:trips-G-H}(iii)--(iv) they have at most one essential crossing. Thus we may consider the case when both $\ell_1$ and $\ell_2$ either begin or end in the flaps. If one of them does not use the $B$-ossicle, the result is again clear, so suppose both of them use the $B$-ossicle and hence travel through $G$ outside of the ear.

    Suppose $\ell_1$ begins in the flaps and $\ell_2$ ends in the flaps. Since $\ell_1'$ and $\ell_2'$ from \Cref{lem:trips-G-H}(iii)--(iv) do not have bad double crossings, the corresponding portions of $\ell_1$ and $\ell_2$ outside the ear do not have bad double crossings. The additional portions of $\ell_1$ and $\ell_2$ are at opposite ends and cannot create bad double crossings. Thus we may suppose $\ell_1$ and $\ell_2$ both begin or both end in the flaps. Without loss of generality, suppose they both begin in the flaps.

    First suppose $\ell_1$ and $\ell_2$ both begin in the same flap, say the flap of $A$. They then travel together through the $C$-ossicle before leaving the ear and continuing on, behaving exactly as the segments $\ell_1'$ and $\ell_2'$ in $H$ starting from the flap of $A'$. By inductive hypothesis, these latter segments do not have bad double crossings, so the same is true of $\ell_1$ and $\ell_2$.

    Finally suppose $\ell_1$ begins in the flap of $A$ and $\ell_2$ begins in the flap of $C$ before traveling outside the ear. They immediately cross in the ear before exiting on opposite sides into the rest of $G$. This latter phase corresponds precisely to the walks taken by $\ell_1'$ and $\ell_2'$ in $H$ from \Cref{lem:trips-G-H}(iii)--(iv). To avoid a bad double crossing, we must show that $\ell_1'$ and $\ell_2'$ do not have an essential crossing in $H$.
    
    Say that $\ell_1'$ begins at some vertex $u'$ of the flap of $A'$ and $\ell_2'$ begins at some vertex $v'$ of the flap of $C'$. Consider the values of $\trip_\bullet(H)(u')$ and $\trip_\bullet(H)(v')$. By \Cref{thm:trip=prom} and \Cref{prop:prom1}, these can be read off from the underlying perfect matching of $H$, which we now describe. Before ear removal, the lobe and flap of $A$ form a nested sequence of $p$ matched edges, and the lobe and flap of $C$ form an adjacent nested sequence of $q$ matched edges. After ear removal, the $p+q$ elements in the lobe of $A'$ and the flap of $C'$ match with the $p+q$ elements in the flap of $A'$ and the lobe of $C'$. In \Cref{fig:arches}, we draw the underlying perfect matching of $H$ schematically starting with this nested sequence of $p+q$ edges, which is followed by an arbitrary perfect matching $M'$. Examining the openers and closers with respect to the positions of $u'$ and $v'$ yields the possible values of $\trip_\bullet(H)(u')$ and $\trip_\bullet(H)(v')$. Since $\ell_1'$ and $\ell_2'$ by assumption enter the rest of $H$, they must end at one of the points of $M'$. Say $\ell_1'$ ends at $u''$ and $\ell_2'$ ends at $v''$.

    We claim that, in the order of \Cref{fig:arches}, $v' < u' < u'' < v''$, in which case the result follows from \Cref{lem:fully-reduced-crossings}(i). These inequalities are clear, except for $u'' < v''$. First note that since $v'$ is in the flap of $C'$, we have $0 \leq \alpha < q$, and since all $q$ elements of the lobe of $C'$ are closers with respect to $u'$, we have $q \leq \beta$. In particular, $\alpha < \beta$. Now the values of $\trip_i(H)(v')$ in $M'$ are the closers of $M'$ from left to right, with $i=\alpha+1, \alpha+2, \ldots$, as in \Cref{fig:arches}. Similarly the values of $\trip_i(H)(u')$ in $M'$ are the openers of $M'$ from left to right, with $i=\beta+1, \beta+2, \ldots$. Since $\alpha < \beta$, it follows that $\trip_i(H)(v')$ and $\trip_{i \pm 1}(H)(v')$ are further to the right than $\trip_i(H)(u')$, so indeed $u'' < v''$, which completes the proof.
\end{proof}

\begin{figure}[htp]
    \centering
    \includegraphics[width=0.75\textwidth]{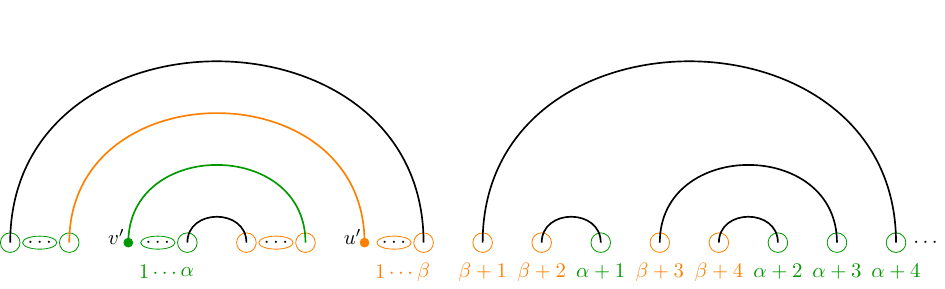}
    \caption{Schematic showing the values of $\trip_i(H)(v')$ (circled in green) and $\trip_i(H)(u')$ (circled in orange) as in the proof of \Cref{thm:fraser-implies-fully-reduced}. Here $i$ is written below the circled openers or closers in the appropriate color.}
    \label{fig:arches}
\end{figure}

\begin{proposition}\label{prop:contracted}
    Fraser graphs are contracted.
\end{proposition}
\begin{proof}
    This is immediate from the fact that all  internal black vertices in a Fraser graph are trivalent, while all internal white vertices are adjacent to the boundary.
\end{proof}

\section{Square faces in hourglass plabic graphs}\label{sec:square_faces}
In this section, we study subgraphs and faces of hourglass plabic graphs. The results of this section hold more generally than in the Pl\"ucker degree $2$ setting, but use facts we have developed about the Pl\"ucker degree $2$ case. We will apply these general results to the Pl\"ucker degree $2$ setting in \Cref{sec:fr_is_fraser}.

\subsection{Subgraphs of hourglass plabic graphs}
We will need the following definition for our study of faces of fully reduced hourglass plabic graphs.

\begin{definition}\label{def:subHPG}
Let $G$ be an hourglass plabic graph and consider a Jordan curve $C$ immersed into the interior of the embedding disk. Assume that $C$ passes through no vertex of $G$. For each edge $e$ of $G$, assume further that $C$ meets $e$ at most once and that any intersection of $C$ and $e$ is transverse. In this case, we define the \emph{sub-hourglass plabic graph} $G|_C$ as follows. By the Jordan--Schoenflies Theorem, continuously deform the interior of $C$ to a disk while deforming $C$ to its boundary circle. The internal vertices and internal edges of $G|_C$ are the images under this deformation of the vertices and edges of $G$ that lie completely in the interior of $C$. The colors of these internal vertices and the multiplicities of these internal edges in $G|_C$ are as the corresponding colors and multiplicities in $G$. 

Finally, for each $m$-hourglass edge $e$ of $G$ having nonempty intersection with $C$, let $v$ be the endpoint of $e$ inside $C$ and let $w$ be the endpoint outside $C$. Now, replace $e$ in $G|_C$ with $m$ consecutive edges of multiplicity $1$ that are incident to the image of $v$ and with their other endpoints being $m$ distinct new consecutive vertices of degree $1$ on the bounding circle $C$; these new boundary vertices are each given the color of $w$. We call these boundary vertices and their incident edges the \emph{claw corresponding to} $e$. The boundary vertices created in this fashion are the only boundary vertices of $G|_C$.
\end{definition}

The class of fully reduced hourglass plabic graphs is important because such objects index our bases in, for example, \Cref{thm:basis} and \cite[Thm.~A]{Gaetz.Pechenik.Pfannerer.Striker.Swanson:4row}. We next observe that the ``fully reduced'' property is hereditary, a useful feature when working with the associated diagrammatic calculus.

\begin{proposition}\label{prop:subHPG}
    If $G$ is a fully reduced hourglass plabic graph and $C$ is a Jordan curve as in \Cref{def:subHPG}, then the sub-hourglass plabic graph $G|_C$ is also fully reduced.
\end{proposition}
\begin{proof}
    First observe that each $\trip_i$-segment of $G|_C$ may be identified with a contiguous piece of a $\trip_i$-segment of $G$. (Here we identify a boundary edge of the claw corresponding to an hourglass edge $e$ with the appropriate strand of $e$.) 
    
    Since $G$ is fully reduced, no $\trip_i$-segment of $G$ has self-intersections and hence the same is true of the truncated $\trip_i$-segments of $G|_C$. 
    Similarly, each bad double crossing in $G|_C$ gives rise to a corresponding bad double crossing in $G$. Hence, the full reducedness of $G$ implies that of $G|_C$.
\end{proof}

\subsection{Fully reduced square faces}
In this subsection, we study square faces of hourglass plabic graphs. We characterize when they are fully reduced in \Cref{thm:square-fully-reduced} and show that square moves preserve trip permutations in \Cref{thm:square-move-preserves-trip-bullet}. We will need these facts in the Pl\"ucker degree two case in \Cref{sec:fr_is_fraser}. Note that our proofs rely on \Cref{thm:trip=prom,thm:fraser-implies-fully-reduced}, but the results in this section hold for general $r$-hourglass plabic graphs.

\begin{definition}
    The edges of a plabic graph decompose the embedding disk into regions. If such a region is not adjacent to the boundary of the disk, we call it a \emph{face}. A \emph{face} of an hourglass plabic graph is a face of the underlying plabic graph. 
\end{definition}

We emphasize that the lacunae between the strands of an hourglass edge are not considered faces.
Note that, by bipartiteness, a face of an hourglass plabic graph without isolated components must be the interior of an even cycle. We call a face a \emph{square} if this cycle has exactly $4$ edges; squares play a special role in the theory of hourglass plabic graphs. Throughout, we assume for simplicity that there are no isolated components.

\begin{definition}
    Let $F$ be a face of an hourglass plabic graph $G$ bounded by vertices $v_1, \dots, v_{2f}$ and edges $e_1,\ldots,e_{2f}$. Define $m(F) \coloneqq \sum_{i=1}^{2f} m(e_i)$. Drawing an immersed Jordan curve $C$ that contains all the vertices and edges of $F$ in its interior, intersects the other edges incident to $v_1, \dots, v_{2f}$ transversally, and does not intersect any other edges or vertices of $G$. We write $\widetilde{F}$ for the hourglass plabic graph $G|_C$ obtained by restricting $G$ as in \Cref{def:subHPG}. Note that $v_1, \dots, v_{2f}$ and $e_1,\ldots,e_{2f}$ are the only internal vertices and edges of $\widetilde{F}$. We say that the face $F$ of $G$ is \emph{fully reduced} if the hourglass plabic graph $\widetilde{F}$ is.
\end{definition}

By \Cref{prop:subHPG}, if $G$ is a fully reduced hourglass plabic graph, then each face $F$ of $G$ is also fully reduced.
The following theorem shows what kinds of squares can appear in fully reduced hourglass plabic graphs.
\begin{theorem}
\label{thm:square-fully-reduced}
    A square $F$ in an $r$-hourglass plabic graph $G$ is fully reduced if and only if $m(F) \leq r$.
\end{theorem}
\begin{proof}
$(\Rightarrow)$
Suppose an $r$-hourglass plabic graph $G$ has a square $F$ with $m(F) > r$. We may assume $G = \widetilde{F}$. We identify a bad double crossing in $G$.

Let the edges of $F$ be $e_1, \dots, e_4$ in clockwise order. Let $v_i$ be the vertex of $F$ incident to $e_i$ and $e_{i+1}$, where we interpret $e_5$ as $e_1$.
Without loss of generality, suppose that $v_1$ is black with simple degree greater than $2$. Let $d \coloneqq r-m(e_1)-m(e_2)$, which is positive.

Let $s_1,\dots,s_d$ be the strands incident to $v_1$ that are not part of $e_1$ or $e_2$ in clockwise order around $v_1$. If $1 \leq i \leq m(e_4)$, then the $\trip_{i+m(e_1)-1}$-segment into $v_1$ along $s_i$ tunnels through the hourglass $e_1$ and travels along $e_4$ to $v_3$ (see \Cref{lem:tunnel}). Similarly, if $d-m(e_3)+1 \leq i \leq d$, then the $\trip_{i+m(e_1)}$-segment into $v_1$ along $s_i$ tunnels through the hourglass $e_2$ and travels along $e_3$ to $v_3$. Both of these conditions on $i$ may be satisfied simultaneously if $m(e_4) \geq d-m(e_3)+1$, which holds since
  \[ m(e_4) + m(e_3) - d - 1 = m(F) - (r+1) \geq 0. \]

We may suppose these $\trip$-segments do not have self-intersections, so they either exit the square at $v_3$ or they continue along the square and exit at $v_2$ or $v_4$, respectively. In any of these cases, they form a bad double crossing.

\medskip
\noindent
$(\Leftarrow)$
    By \Cref{thm:fraser-implies-fully-reduced}, all Fraser graphs are fully reduced, so by \Cref{prop:subHPG} any square appearing in a Fraser graph is fully reduced. Hence to prove this direction of the theorem, it suffices to realize the square $F$ as a subconfiguration of some Fraser graph.

    Let the square $F$ be oriented and indexed as on the left of \Cref{fig:square-move}, but with $m_1 + m_2 + m_3 + m_4 \leq r$. Set $\epsilon = r - (m_1 + m_2 + m_3 + m_4)$. Then one may check that $\widetilde{F}$ arises as a Fraser graph from the weighted triangulation
    \begin{center}
    \includegraphics[]{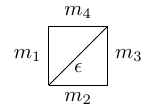}.
    \end{center}
    This triangulation arises in turn from a standard tableau of rectangular shape $r \times 2$, which we omit.
\end{proof}

\begin{remark}
    If a web $W$ contains a square face $F$ with $m(F)>r$, one may iteratively apply the \emph{square relation} \cite[Eq.~2.10]{Cautis-Kamnitzer-Morrison} to write $[W]_q$ as a sum of invariants of webs containing no such face. In the case $r=3$, \Cref{thm:square-fully-reduced} recovers Kuperberg's famous \emph{non-elliptic condition} \cite{Kuperberg} on the non-existence of square faces.
\end{remark}

\begin{remark}
    Fraser exhibits in \cite[Eq.~2.4]{Fraser-2-column} a plabic graph $G$ which gives rise to two invariants, neither of which is dual canonical. This $G$ can be interpreted as an hourglass plabic graph in two ways, and neither of these is fully reduced, since both have a square face violating the condition from \Cref{thm:square-fully-reduced}.
\end{remark}

\begin{theorem}
\label{thm:square-move-preserves-trip-bullet}
    For any $r$-hourglass plabic graph, the square move preserves $\trip_{\bullet}$.
\end{theorem}
\begin{proof}
Let $G$ and $G'$ be hourglass plabic graphs related by a square move at faces $F$ and $F'$, respectively. By \Cref{thm:square-fully-reduced}, the restrictions $\widetilde{F}$ and $\widetilde{F}'$ are Fraser. Moreover, the proof of n \Cref{thm:square-fully-reduced} shows that they arise as Fraser hourglass plabic graphs from the same undissected square and hence from the same standard Young tableau. By \Cref{thm:trip=prom}, they therefore also have the same trip permutations. It follows that $G$ and $G'$ then also have the same trip permutations.
\end{proof}

\section{Fully reduced graphs are Fraser}
\label{sec:fr_is_fraser}
We now complete the characterization of contracted fully reduced hourglass plabic graphs of Pl\"{u}cker degree two by showing that they are exactly the Fraser graphs.

\begin{theorem}
\label{thm:fully-reduced-implies-fraser}
Let $G$ be a contracted, fully reduced hourglass plabic graph of Pl\"{u}cker degree two. Then $G$ is Fraser.
\end{theorem}

We refer the reader to \cite{Fomin.Williams.Zelevinsky, Postnikov-arxiv} for background on plabic graphs.
We say a plabic graph is \emph{weakly bipartite} if all pairs of adjacent internal vertices have opposite colors. The set $\aexc(\pi)$ of \emph{antiexcedances} of a permutation $\pi$ is $\{i \mid \pi^{-1}(i)> i\}$.
We need the following fact about plabic graphs.

\begin{lemma}
\label{lem:structure-of-underlying-plabic}
       Let $H$ be a contracted, weakly bipartite, reduced plabic graph with $\left|\aexc(\trip(H))\right|=2$. Then all internal black vertices have degree 3 and all internal white vertices are adjacent to a boundary vertex.
\end{lemma}
\begin{proof}
    Let $\pi: i \mapsto i+2 \mod n$ denote the permutation with two antiexcedances which is the unique maximum element under \emph{circular Bruhat order} (see \cite[\S 17]{Postnikov-arxiv}). \Cref{fig:flower} gives a contracted, weakly bipartite, reduced plabic graph $G$ with trip permutation $\pi$ which has the desired properties.

\begin{figure}[htb]
    \centering
    \includegraphics[scale=0.7]{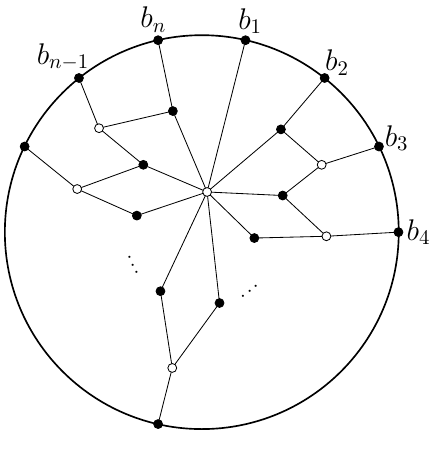}
    \caption{The reduced plabic graph with trip permutation $i \mapsto i+2 \bmod n$ used in the proof of \Cref{lem:structure-of-underlying-plabic}.}
    \label{fig:flower}
\end{figure}

    Now let $\sigma=\trip(H)$, and consider a saturated chain $\sigma = \sigma_0 \lessdot \sigma_1 \lessdot \cdots \lessdot \sigma_{y} = \pi$ in circular Bruhat order. Suppose for $j \in \{1, \dots, y\}$ that we have a contracted, weakly bipartite, reduced plabic graph $H_{j}$ for $\sigma_{j}$ having all internal black vertices of degree 3 and all internal white vertices adjacent to a boundary vertex. Then we can obtain $H_{j-1}$ for $\sigma_{j-1}$ having these same properties, as follows: since $\sigma_{j-1} \lessdot \sigma_j$, there exists an internal edge of $H_{j}$ which can be deleted, resulting in a reduced plabic graph $H'_{j-1}$ for $\sigma_{j-1}$ (see \cite[Thm.~17.8]{Postnikov-arxiv}). Let $v_b$ and $v_w$ be the black and white endpoints of this edge, respectively. Since $v_b$ has degree 3 in $H_j$, it has degree 2 in $H'_{j-1}$, and so we may remove it and contract its two neighboring white vertices. If $v_w$ also has degree 2 in $H'_{j-1}$, we may remove it; since one of its neighboring vertices is a boundary vertex, no contraction is necessary in order to maintain weak bipartiteness. Call the resulting plabic graph $H_{j-1}$.

    Let $H_0$ be the plabic graph for $\sigma$ obtained by iterating the above construction, beginning with $H_{y}=G$. Since $\trip(H_0)=\trip(H)$, and since both are reduced, by \cite[Thm.~13.4]{Postnikov-arxiv} $H_0$ and $H$ are move-equivalent. Since both are contracted and weakly bipartite, they are in fact connected by a sequence of \emph{bipartite square moves} (see \Cref{fig:contracted-sq-move}). These moves preserve the properties of all internal black vertices having degree 3 and all white vertices being incident to the boundary. Thus $H$, like $H_0$, has these properties.
\end{proof}

\begin{figure}[hbt]
    \centering
    \includegraphics[scale=1]{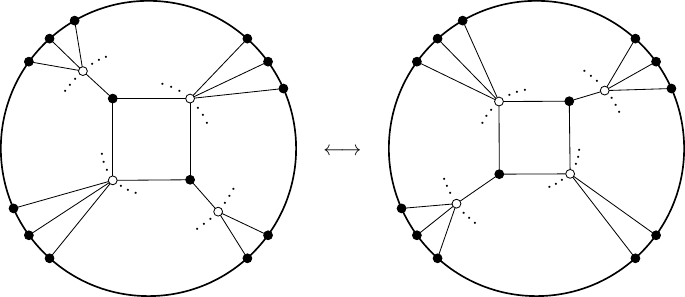}
    \caption{A \emph{bipartite square move} of a plabic graph, obtained by composing the square move of \Cref{fig:square-move} with contraction moves. All adjacencies of the black vertices are shown in the figure, but the white vertices may be adjacent to additional vertices.}
    \label{fig:contracted-sq-move}
\end{figure}

\begin{lemma}\label{lem:white_minus_black}
    Let $G$ be an $r$-hourglass plabic graph of Pl\"ucker degree $p$ with $w$ internal white vertices and $b$ internal black vertices.
    Then $w-b=p$.
\end{lemma}
\begin{proof}
By bipartiteness, the total degree $rw$ of the white vertices must equal the total degree $rp+rb$ of the black vertices. 
\end{proof}

\begin{lemma}\label{lem:w-b-equals-aexc}
    Let $H$ be a bipartite reduced plabic graph with $w$ white and $b$ black internal vertices, then $\left|\aexc(\trip(H))\right|=w-b$.
\end{lemma}
\begin{proof}
    An \emph{almost perfect matching} of $H$ is a subset $M$ of the edges such that all internal vertices, and some number $k$ of the boundary vertices, are incident to exactly one edge from $M$.
    Since $H$ is reduced, it admits an almost perfect matching (called a \emph{partial matching} in \cite{Postnikov-arxiv}) by \cite[Lem.~11.10 \& Thm.~12.7]{Postnikov-arxiv}. Since $H$ is bipartite, we must have $w-b=k$ for any such matching.
    On the other hand, \cite[\S 16]{Postnikov-arxiv} implies that $\left|\aexc(\trip(H))\right|=k$.
\end{proof}

\begin{lemma}
\label{lem:faces-are-squares}
    Let $G$ be a connected, contracted, fully reduced hourglass plabic graph of Pl\"{u}cker degree two. Then all internal faces of $G$ are squares. Furthermore, the set of boundary vertices incident to each internal white vertex is cyclically consecutive. Under the induced cyclic order on the internal white vertices, any consecutive pair of such vertices share a unique neighbor.
\end{lemma}
\begin{proof}
    By \Cref{prop:underlying-plabic-is-reduced}, $H \coloneqq \widehat{G}$ is reduced, and, since $G$ is contracted and bipartite, so is $H$. Let $w$ and $b$ denote the number of internal white and black vertices of $H$ (equivalently, of $G$), respectively. 

    Since $w-b=2$ by \Cref{lem:white_minus_black}, we have that $\left|\aexc(\trip(H))\right|=2$ by \Cref{lem:w-b-equals-aexc}.
     Thus $H$ satisfies all of the hypotheses of \Cref{lem:structure-of-underlying-plabic}, so we may conclude that all internal black vertices have degree 3 and all internal white vertices are adjacent to the boundary.

    View $H$ as a planar embedded graph, ignoring its boundary circle (but retaining the boundary vertices). Let $v,e,$ and $f$ denote the number of vertices, edges, and faces, respectively, so that $v-e+f=1$. We have that $v=n+w+b=2r+w+b$ and that $e=2r+3b$, since $H$ is bipartite with $2r$ black vertices of degree 1 and $b$ black vertices of degree $3$. Again using the fact that $w-b=2$, we see that $f=b-1$. Let $a$ denote the average number of black vertices incident to a face of $H$. Each internal black vertex is incident to at most 3 faces. For each white vertex $u$, consider a boundary vertex $b_i$ adjacent to $u$ such that $b_{i+1}$ is adjacent to another white vertex $u' \neq u$, where indices are taken modulo $n$. Since $H$ is connected, if we were to add the edge between $b_i$ and $b_{i+1}$, there would be a face $F$ containing edges $b_i b_{i+1}, ub_i$ and $u'b_{i+1}$; furthermore, $F$ would contain some internal black vertex $x$ along the path from $u$ to $u'$. Since $H$ does not in fact contain the edge $b_i b_{i+1}$, there is no such face $F$, and we have found an incidence between $x$ and the unbounded region of $H$. This argument produces $w=b+2$ incidences between internal black vertices and the unbounded region of $H$. We now have
    \begin{equation}
    \label{eq:face-inequality}
                    3b-(b+2) \geq af = a(b-1),
    \end{equation}
    so $a \leq 2$. Since $H$ is reduced, it contains no $2$-gons; since it is bipartite, the smallest possible face is a square, so $a \geq 2$. Thus we must in fact have that $a=2$ and that all faces of $H$ (and hence of $G$) are squares.

    If the set of boundary vertices incident to some internal white vertex $u$ was not cyclically consecutive, then there would be multiple choices for the boundary vertex $b_i$. The analogue of \eqref{eq:face-inequality} would then give $a<2$, an impossibility. Thus the internal white vertices have a cyclic order induced by the cyclic order on the boundary vertices adjacent to them. If two consecutive white vertices $u$ and $u'$ did not share a neighbor, then the path considered above between $u$ and $u'$ would contain at least two internal black vertices, again yielding the contradiction $a<2$. Finally, $u$ and $u'$ cannot share more than one such neighbor, otherwise the internal black vertex on the would-be face $F$ would have no third white vertex to connect to, contradicting the fact that internal black vertices in $G$ have simple degree 3. 
\end{proof}

\begin{proof}[Proof of \Cref{thm:fully-reduced-implies-fraser}]  
    We first consider the case in which $G$ is disconnected. By degree counting as in \Cref{lem:white_minus_black}, there must be two connected components, each with $w-b=1$. Thus $G$ is the disjoint union of two contracted fully reduced hourglass plabic graphs of Pl\"{u}cker degree one. By \Cref{prop:plucker-degree-one-implies-star}, each of these is a star graph (see \Cref{fig:star-graph}). If the two connected components contain boundary vertices $\{b_1,\ldots,b_r\}$ and $\{b_{r+1},\ldots,b_{2r}\}$, respectively, then this stitched graph is the image under $\mathcal{F}$ of the \emph{column superstandard} tableau containing entries $1, \dots, r$ in its left column. Thus $G$ is Fraser.

    Now assume that $G$ is connected, so that $G$ satisfies the hypotheses of \Cref{lem:faces-are-squares}. We will use $G$ to construct a weighted triangulation $\mathfrak{t}$ of a $w$-gon $P$; see \Cref{fig:polygon-construction}. Identify the vertices of $P$ with the internal white vertices of $G$, ordered around $P$ consistently with the cyclic order on these vertices (guaranteed to exist by \Cref{lem:faces-are-squares}). For each internal square $F$ of $G$, draw a diagonal of $P$ between the two white vertices contained in $F$. By the proof of \Cref{lem:faces-are-squares}, there are $b-1=w-3$ such squares, and hence the same number of diagonals. At most one diagonal is drawn between each pair of vertices of $P$, since if two white vertices were both contained in two squares of $G$, there would be an internal black vertex of degree 2, contradicting \Cref{lem:structure-of-underlying-plabic}. Thus we have defined a dissection of the $w$-gon $P$ with $w-3$ diagonals. Such a dissection must be a triangulation. Finally, give the diagonal associated to a square $F$ the weight $r-m(F)$; by \Cref{thm:square-fully-reduced} and the fact that $G$ is fully reduced, this weight is nonnegative. To each boundary edge $e$ of $P$, assign weight equal to the multiplicity of the hourglass opposite $e$ incident to the unique black vertex contained in the triangle bounded by $e$.

    \begin{figure}[htbp]
        \centering
        \includegraphics[scale=1.55]{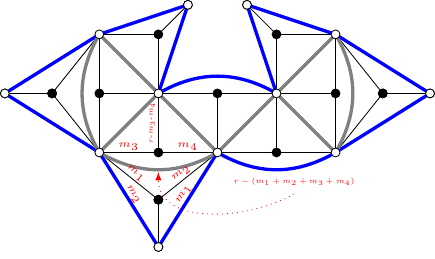}
        \caption{The polygon $P$ (blue) with its weighted triangulation $\mathfrak{t}$ (gray) constructed from the hourglass plabic graph $G$ (black) as in the proof of \Cref{thm:fully-reduced-implies-fraser}.}
        \label{fig:polygon-construction}
    \end{figure}

    We now argue that $G$ is the image of the weighted triangulation $\mathfrak{t}$ under Fraser's map $\mathscr{W}$. All that remains to be verified is that the edge weights in $\mathfrak{t}$ are correct. 
    The weights of the boundary edges in each ear agree with Fraser's construction by definition. We may remove an ear from the graph and polygon, producing a new triangulation $\mathfrak{t}'$ with one fewer triangle, while changing the weight of the non-boundary edge of the removed triangle from $r-(m_1+m_2+m_3+m_4)$ to $r-(m_3+m_4)$ (see \Cref{fig:polygon-construction}). This operation preserves the total weight of the triangulation. By induction, we see that weights of $\mathfrak{t}$ are as prescribed by Fraser's map.
\end{proof}

\begin{remark}
    The only properties of fully reduced graphs that we have relied upon in the proof of \Cref{thm:fully-reduced-implies-fraser} are that square faces $F$ in such graphs have $m(F) \leq r$ (\Cref{thm:square-fully-reduced}) and that the underlying plabic graph is reduced (\Cref{prop:underlying-plabic-is-reduced}).
\end{remark}

\begin{proof}[Proof of Theorems 1.1 and 1.2]
\Cref{thm:fraser-implies-fully-reduced,thm:fully-reduced-implies-fraser} together imply that the set of graphs resulting from Fraser's construction over all $T\in\SYT(r\times 2)$ are precisely the contracted fully reduced $r$-hourglass plabic graphs of standard type and Pl\"ucker degree two. \Cref{prop:Fraser_web_invariant} states that the  tensor invariant of $\mathcal{F}(T)$ does not depend on the choices involved in the construction, thus all such graphs in a square move equivalence class yield the same tensor invariant. Since rotation of a Fraser graph produces another Fraser graph, the basis is rotation-invariant, completing the proof of \Cref{thm:basis}.   

\Cref{thm:trip=prom} establishes the $\trip_{\bullet}=\prom_{\bullet}$ property of \Cref{thm:bijection}. Furthermore, since promotion applied to $T\in\SYT(r\times 2)$ rotates the underlying perfect matching $\mathscr{M}(T)$ and evacuation reflects $\mathscr{M}(T)$, we have that promotion and evacuation of tableaux correspond to rotation and reflection of hourglass plabic graphs, completing the proof of \Cref{thm:bijection}.
\end{proof}

\section{Pl\"{u}cker degree one}\label{sec:1column}

In this short section we show that fully reduced hourglass plabic graphs also provide a web basis in Pl\"{u}cker degree one. \Cref{prop:plucker-degree-one-implies-star} is used in the proof of \Cref{thm:fully-reduced-implies-fraser}. (Note that nothing in this section relies on results of \Cref{sec:fr_is_fraser}.)

The proof of \Cref{prop:plucker-degree-one-implies-star} relies on the difficult \cite[Thm.~13.4]{Postnikov-arxiv}; we are unaware of a simple direct proof.

\begin{proposition}
\label{prop:plucker-degree-one-implies-star}
    The unique contracted fully reduced $r$-hourglass plabic graph of Pl\"{u}cker degree one is the star graph (see \Cref{fig:star-graph}).
\end{proposition}
\begin{proof}
    It is easy to see that the star graph $S_r$ is indeed a contracted fully reduced $r$-hourglass plabic graph of Pl\"{u}cker degree one. 
    
    Now, let $G$ be any contracted fully reduced $r$-hourglass plabic graph of Pl\"{u}cker degree one. If $r=1$, there is only one possible hourglass plabic graph, a star, so we are done. Thus assume that $r \geq 2$. By \Cref{prop:underlying-plabic-is-reduced}, $H\coloneqq  \widehat{G}$ is reduced, and, since $G$ is of standard type, none of the internal vertices adjacent to a boundary vertex are leaves; thus $\trip(H)$ has no fixed points. Since $w-b=1$, we have $\left|\aexc(\trip(H))\right|=1$. The only possibility is $\trip(H)=23\ldots r 1$. Note that $S_r$ is a reduced plabic graph with trip permutation $23\ldots r 1$. Thus by \cite[Thm.~13.4]{Postnikov-arxiv}, $H$ is move-equivalent to $S_r$. Since $S_r$ admits no moves other than (un)contraction moves, and since $H$ is contracted, we conclude that $H$ and $G$ are both star graphs.
\end{proof}

\begin{figure}[htbp]
    \centering
    \includegraphics[scale=0.7]{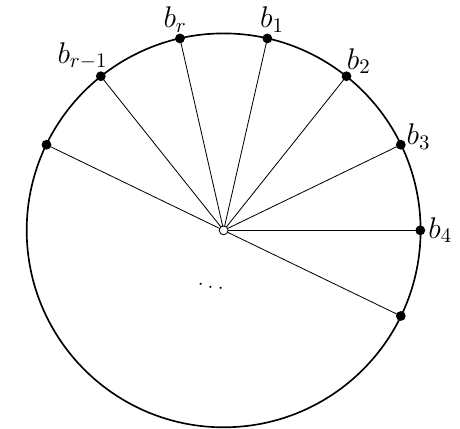}
    \caption{The star graph $S_r$, referred to in \Cref{prop:plucker-degree-one-implies-star}.}
    \label{fig:star-graph}
\end{figure}

Note that for each $r$, there is a unique standard Young tableau $T_r$ of shape $r \times 1$. The theorem below follows trivially from \Cref{prop:plucker-degree-one-implies-star}, but we state it in a form paralleling \Cref{thm:basis,thm:bijection} in order to highlight the uniformity of the framework of fully reduced hourglass plabic graphs.

\begin{theorem}\label{thm:1column}
The collection $\mathscr{B}_q^{r}$ of tensor invariants of fully reduced $r$-hourglass plabic graphs of Pl\"{u}cker degree one is a rotation-invariant web basis for  $\Inv_{U_q(\fsl_r)}(V_q^{\otimes r})$. In addition, we have $\trip_{\bullet}(S_r)=\prom_{\bullet}(T_r)$.
\end{theorem}
\begin{proof}[Proof of \Cref{thm:1column}]
By \Cref{prop:plucker-degree-one-implies-star}, we have $\mathscr{B}_q^r = \{[S_r]_q\}$. Since $|\SYT(r \times 1)|=1$, we see that indeed this is a basis for the 1-dimensional space $\Inv_{U_q(\fsl_r)}(V_q^{\otimes r})$; it is rotation-invariant since $S_r$ is invariant under rotation. Finally, it is easy to check directly that $\trip_i(S_r)=\prom_i(T_r)$ are both given by $j \mapsto j+i \mod{r}$.
\end{proof}

\section{Combinatorial applications of the web basis}
\label{sec:combinatorial_consequences}
In this section, we discuss some combinatorial consequences of our web basis. We first isolate a particular extreme
move-equivalence class relating to the Tamari lattice. We then re-prove a cyclic sieving result. 
\subsection{The Tamari lattice}
\label{sec:tamari}
As shown in the proof of \cite[Prop.~1.13]{Fraser-2-column} and depicted in \Cref{fig:FraserMap,fig:triflip}, square moves correspond to flips of weight $0$ diagonals in the triangulation produced in the construction of $\mathcal{F}(T)$. For the ``superstandard'' SYT constructed by filling the numbers in order left to right, then top to bottom, this construction gives diagonals all of weight zero, thus the move equivalence class is counted by the Catalan numbers and is indeed equivalent to the Tamari lattice; see \Cref{prop:superstandard} below. In \cite[Prop 8.2]{Gaetz.Pechenik.Pfannerer.Striker.Swanson:4row}, we showed the square move equivalence class of the hourglass plabic graph corresponding to the superstandard $4\times d$ SYT is in bijection with the set of $d\times d$ alternating sign matrices. We also found plane partitions in an $a\times b\times c$ box as a move-equivalence class of hourglass plabic graphs for a certain $4$-row fluctuating tableau \cite[Prop 8.3]{Gaetz.Pechenik.Pfannerer.Striker.Swanson:4row}. (In both cases, the moves correspond to the cover relations in the lattice.) Thus, the proposition below detects yet another important class of combinatorial objects within the hourglass plabic graph framework.

\begin{proposition}
\label{prop:superstandard}
    Let $T$ be the ``superstandard'' tableau in $\SYT(r\times 2)$ constructed by filling the numbers in order left to right, then top to bottom. The square move equivalence class of $\mathcal{F}(T)$ is in bijection with triangulations of an $r$-gon connected by diagonal flips, i.e.\ the Tamari lattice. Thus, it is counted by the $(r-2)$-nd Catalan number: $\frac{1}{r-1}\binom{2(r-2)}{r-2}$.
\end{proposition}
\begin{proof}
    The matching $\mathscr{M}(T)$ consists of arcs connecting $2i-1$ and $2i$ for $1\leq i\leq r$. Thus the claw sets are given as $C_j=\{2i-1,2i\}$, so there are $r$ claw sets, each of cardinality $2$. As there are no arcs in $\mathscr{M}(T)$ between claw sets, we have that $\dis(\mathscr{M}(t))$ is the weighted dissection of an $r$-gon in which there are no internal edges. Then in the construction of the triangulation $\mathfrak{t}(\dis(\mathscr{M}(T)))$, we have a choice amongst all triangulation of the $r$-gon; these are well-known to be counted by the $(r-2)$-nd Catalan number $\frac{1}{r-1}\binom{2(r-2)}{r-2}$ (see e.g.~\cite[p.15]{Stanley-catalan}). The poset constructed via diagonal flips is a lattice, called the \emph{Tamari lattice} (see e.g.~\cite[p.16,119]{Stanley-catalan}). 
\end{proof}

\subsection{Rectangular cyclic sieving}

 A major motivation for the dynamical algebraic combinatorics community to study webs involves the \emph{cyclic sieving phenomenon}     \cite{Reiner-Stanton-White}. Briefly, the orbit structure of a cyclic group action on a finite set is encoded in evaluations of a polynomial at complex roots of unity. A celebrated result of Rhoades \cite{Rhoades} provides such a description for the promotion  action $\promotion$ on $\SYT(r \times d)$ using the \emph{$q$-hook length polynomial}
  \[ f^\lambda(q) \coloneqq \frac{[n]_q!}{\prod_{c \in \lambda} [h_c]_q}. \]
See \cite[Cor.~7.21.6]{Stanley:EC2}, \cite{Rhoades}, or \Cref{thm:csp} for details.

Rhoades' proof uses Kazhdan--Lusztig theory. The $3$-row case was re-proven in \cite{Petersen-Pylyavskyy-Rhoades} using Kuperberg's non-elliptic web basis \cite{Kuperberg}. The authors' recent $4$-row web basis extended this alternate approach to Rhoades' result to the $4$-row case as well \cite[Cor.~8.1]{Gaetz.Pechenik.Pfannerer.Striker.Swanson:4row}.

We now sketch how Rhoades' result follows from the existence of a rotation-invariant fully reduced web basis satisfying our combinatorial framework from \Cref{thm:basis,thm:bijection}, which at present includes $\leq 4$ rows or $\leq 2$ columns. As a practical matter, cyclic sieving proofs frequently involve carefully tracking negative signs. In our argument, this is done through the fully reduced condition on ``middle'' $\trip$-strands.

\begin{theorem}[See {\cite[Thm.~1.3]{Rhoades}}]{\label{thm:csp}}
    Suppose there is a basis $\mathscr{B}_q^{rd}$ of $\Inv_{U_q(\fsl_r)}(V_q^{\otimes rd})$ consisting of tensor invariants of contracted fully reduced $r$-hourglass plabic graphs. Further suppose there is a bijection from $\SYT(r \times d)$ to these basis graphs which intertwines promotion and rotation.
    
    In this case, the number of elements in $\SYT(r \times d)$ fixed by $\promotion^k$ is $f^\lambda(\zeta^k)$ where $\lambda = (d, \ldots, d) = (d^r)$, $n=rd$, and $\zeta = e^{2\pi i/n}$.
\end{theorem}

\begin{proof}
    (Sketch.) Let $\sigma = (1\,2\,\cdots\,n) \in S_n$. By Springer's theory of regular elements \cite[Prop.~4.5]{Springer}, we have $\chi^\lambda(\sigma^k) = (-1)^{k(r-1)} f^\lambda(\zeta^k)$, where $\chi^\lambda$ is the character of the Specht module $S^\lambda \cong \Inv(V^{\otimes n})$. The action of $\sigma$ corresponds to web rotation. When $r$ is odd, the sign $(-1)^{k(r-1)}$ disappears, and the result follows from the assumed $\promotion$-equivariance.

    When $r$ is even, more care is needed. In order to completely interpret $r$-hourglass plabic graphs as tensor invariants, we must fix a global sign through the use of \emph{tags}; see \cite[\S2.1]{Gaetz.Pechenik.Pfannerer.Striker.Swanson:4row} for detailed discussion. Equivalently, this requires choosing a clockwise linear order on hourglass edges incident to each interior vertex, where two such orders related by cycling a single $m$-hourglass edge are related by the sign $(-1)^{m(r-m)}$. When $r$ is odd, these signs are all $1$, and tags may be ignored. To handle the even $r$ case, we choose a tagging convention as follows.
    
    First, we extend the model slightly by allowing ``virtual'' tags between strands of a single hourglass. Moving a possibly virtual tag over a single strand results in a sign of $(-1)^{r-1}$, which is consistent with the above.
    
    Since $r$ is even, $\trip_{r/2}$ is a fixed-point-free involution, so the corresponding segments form a perfect matching. Fix an interior vertex $v$. Since $v$ is $r$-valent, there are $r/2$ $\trip_{r/2}$-segments passing through $v$. Some of these segments may be trivial and loop inside a single hourglass of high multiplicity. Since the graph is fully reduced, the non-trivial segments through $v$ reach the boundary at distinct points without double crossing. At least one such segment exists since the graph is contracted and has standard type. Such segments carve the web into regions. Place the virtual tag of $v$ in the region containing the base face.

    If we rotate such a tagged basis web one step, the tags rotate as well. Adjusting them to fit the above tagging convention will involve a sign, which we show is $-1$ and which will complete the proof of the result. Follow the $\vec{\trip}_{r/2}$-segment which is rotated from starting at $1$ to starting at $n$. The interior vertices encountered are precisely those whose tags need to be adjusted, each by moving past a single strand. Since the hourglass plabic graph is bipartite with black boundary vertices, there are an odd number of such interior vertices along this segment, and the theorem follows.
\end{proof}

\section*{Acknowledgements}
Pechenik and Pfannerer were partially supported by Pechenik's Discovery Grant (RGPIN-2021-02391) and Launch Supplement (DGECR-2021-00010) from the Natural Sciences and Engineering Research Council of Canada. Pfannerer was also partially supported by the Austrian Science Fund (FWF) P29275, Olya Mandelshtam's Discovery Grant (RGPIN-2021-02568), and was a recipient of a DOC Fellowship of the Austrian Academy of Sciences. Striker was partially supported by Simons Foundation gift MP-TSM-00002802 and NSF grant DMS-2247089. Swanson was partially supported by NSF grant DMS-2348843.

We thank Chris Fraser for his helpful comments and Dutch Hansen for writing valuable computer code. Work on this project was carried out during our visits to North Dakota State University (partially supported by NSF DMS-2247089) and to the Institute for Computational and Experimental Research in Mathematics in Providence, RI (through the Collaborate@ICERM program supported by NSF DMS-1929284). We are grateful for the excellent working environments provided by both of these institutions. We also thank the anonymous referee for their very insightful comments and suggestions.

\bibliographystyle{amsalphavar}
\bibliography{main}
\end{document}